\documentclass[sn-mathphys,Numbered]{sn-jnl}%

\usepackage{graphicx}%
\usepackage{multirow}%
\usepackage{amsmath,amssymb,amsfonts}%
\usepackage{amsthm}%
\usepackage{mathrsfs}%
\usepackage[title]{appendix}%
\usepackage{xcolor}%
\usepackage{textcomp}%
\usepackage{manyfoot}%
\usepackage{booktabs}%
\usepackage{algorithm}%
\usepackage{algorithmicx}%
\usepackage{algpseudocode}%
\usepackage{listings}%
\usepackage{mathtools}
\usepackage{pinlabel}
\usepackage{xspace}
\usepackage{comment}
\usepackage{accents}
\usepackage{enumitem}
\usepackage{caption}
\usepackage{subcaption}

\newcommand{\X}{\Lambda}
\newcommand{\Y}{M}
\newcommand{\Z}{N}
\newcommand{\x}{\lambda}
\newcommand{\y}{\mu}
\newcommand{\z}{\nu}

\newcommand{\expleend}{\hfill\vbox{\hrule height0.6pt\hbox{%
\vrule height1.3ex width0.6pt\hskip0.8ex%
\vrule width0.6pt}\hrule height0.6pt%
}}

\newcommand{\bH}{\mathbb{H}}
\newcommand{\bG}{\mathbb{G}}
\newcommand{\ints}{\mathbb Z}
\DeclareMathOperator{\ba}{\backslash}
\DeclareMathOperator{\res}{|}
\DeclareMathOperator{\con}{/}
\newcommand\sbullet{\mathbin{\vcenter{\hbox{\scalebox{0.4}{$\bullet$}}}}}
\newcommand{\dott}[1]{\accentset{\sbullet}{#1}}
\newcommand{\hatt}[1]{\widehat{#1}}
\newcommand{\barr}[1]{\overline{#1}}
\newcommand{\vect}[1]{\mathbf{#1}}

\DeclareMathOperator{\barrast}{\overline{\ast}}

\newcommand{\sk}[1]{#1}
\newcommand{\e}[1]{\theta({#1})}
\newcommand{\ez}[2]{\theta_{{#1}}({#2})}

\usepackage[normalem]{ulem}

\newtheorem{theorem}{Theorem}%
\newtheorem{proposition}[theorem]{Proposition}%

\newtheorem{corollary}[theorem]{Corollary}

\theoremstyle{definition}%
\newtheorem{example}[theorem]{Example}%
\newtheorem{remark}[theorem]{Remark}%
\newtheorem{notation}[theorem]{Notation}

\newtheorem{definition}[theorem]{Definition}%

\begin{document}

\title[Tensor products of multimatroids]{Tensor products of multimatroids and a Brylawski-type formula for the transition polynomial}

\author[1]{\fnm{Iain} \sur{Moffatt}}\email{iain.moffatt@rhul.ac.uk}

\author[2]{\fnm{Steven} \sur{Noble}}\email{s.d.noble@leeds.ac.uk}

\author[1]{\fnm{Maya} \sur{Thompson}}\email{mayathompson.math@gmail.com}

\affil[1]{\orgdiv{Department of Mathematics}, \orgname{Royal Holloway, University of London}, \orgaddress{ \city{Egham}, \postcode{TW20~0EX}, \country{United Kingdom}}}

\affil[2]{\orgdiv{School of Computer Science}, \orgname{University of Leeds}, \orgaddress{\street{Woodhouse}, \city{Leeds}, \postcode{LS2~9JT}, \country{United Kingdom}}}

\abstract{ 
Brylawski's tensor product formula expresses the Tutte polynomial of the tensor product of two graphs or matroids in terms of Tutte polynomials arising from the tensor factors.
Analogous tensor product formulas are known for the Bollob\'as--Riordan and transition polynomials of graphs embedded in surfaces. 
We show that these formulas are instances of a more general result for multimatroids by
giving a tensor product formula for the multimatroid transition polynomial and showing that Brylawski's formula and its topological analogues arise from it. 
Along the way we provide formulas for the transition polynomial of the two-sum and star product of multimatroids.
}



\maketitle

\bigskip

\section{Introduction}
Let $G$ and $H$ be graphs, $G$ loopless and bridgeless, and 
$H$ with a distinguished edge $e$ that is neither a loop nor a bridge. The tensor product $G\otimes_e H$ is formed by, for each edge of $G$, identifying that edge with $e$ in a copy of $H$ then deleting the  edge formed by the identification. So each edge of $G$ is replaced with a copy of $H \ba e$. \emph{Brylawski's tensor product formula}~\cite{Brylawski_2010}  expresses the Tutte polynomial of $G\otimes_e H$ in terms of Tutte polynomials coming from $G$ and $H$:
\[
T(G \otimes_e H;x,y) = \alpha^{n(G)} \beta^{r(G)}
T\Big(G;\tfrac{T(H\ba e ;x,y)}{\beta},\tfrac{T(H\con e;x,y)}{\alpha} \Big),
\] 
where $\alpha$ and $\beta$ are  polynomials in $\mathbb{Z}[x,y]$ uniquely defined by the system of linear equations
\begin{align*}
(x-1)\alpha + \beta &= T(H\ba e;x,y), \\
\alpha + (y-1)\beta &= T(H\con e;x,y).
\end{align*}
Brylawski's formula has been extended to graphs embedded in surfaces in two different ways: in~\cite{Huggett_2011} where a tensor product formula for the Bollob\'as--Riordan polynomial was given, and in~\cite{ELLIS_MONAGHAN_2014} which gave a tensor product formula for the topological transition polynomial.

In the classical case, Brylawski proved the  tensor product formula for the Tutte polynomial 
in the more general setting of matroids  rather than just graphs (as is stated in  Corollary~\ref{cor:brymat} below).
 All of the structure required for the result for graphs resides in the cycle matroids of the graphs. 
Here we address the problem of finding the correct framework for the tensor product formulas for graphs embedded in surfaces. 
A reader familiar with topological analogues of the Tutte polynomial might expect this framework to be given by delta-matroids. Delta-matroids~\cite{MR904585} form a generalisation of matroids where bases may have different cardinalities, see~\cite{CMNR1,CMNR2} for further background. However, we shall see that the natural setting is in fact Bouchet's tight multimatroids~\cite{MM1zbMATH01116184, MR1845490}, which form a generalisation of delta-matroids where one may have an arbitrary (but finite) number of minor edge operations, instead of only deletion and contraction.

In this paper we introduce tensor products of tight multimatroids. This operation is compatible with the graph, matroid and embedded graph tensor products. Our main result is a generalisation  of Brylawski's formula to the transition polynomial of multimatroids, given here as Theorem~\ref{thm:main}.   From this we deduce  tensor product formulas for delta-matroids, including the recovery of Brylawski's result for matroids. We also deduce an extension of Ellis-Monaghan and Moffatt's tensor product formula~\cite{ELLIS_MONAGHAN_2014} for the topological transition polynomial by relaxing the restriction on the type of edges that the tensor product acts on.

\medskip

This paper is structured as follows. Section~\ref{sec:multim} provides an overview of multimatroids. Section~\ref{sec:tensors} discusses the two-sum operation for multimatroids and introduces the tensor product. In Section~\ref{sec:isotropic}, we describe related constructions including Bouchet's isotropic systems~\cite{MR919874}. Section~\ref{sec:results} contains our main result, Theorem~\ref{thm:main}, giving a version of Brylawski's formula for the multimatroid transition polynomial. 
Along the way, we give a formula for the multimatroid transition polynomial of the two-sum of multimatroids in Theorem~\ref{thm:trans2sum} and of the star product of multimatroids in Theorem~\ref{thm:star}.
 In Section~\ref{sec:delta} we apply our main result to deduce tensor product formulas for delta-matroid and matroid polynomials. Finally, in Section~\ref{sec:rg} we review a connection between graphs embedded in surfaces and tight 3-matroids and deduce a tensor product formula for the topological transition polynomial. To aid the digestion of our results, we provide plenty of  examples and we do not assume any familiarity with multimatroids.

\section{Multimatroids}\label{sec:multim}
Multimatroids were introduced by Bouchet in~\cite{MM1zbMATH01116184} as a simultaneous generalization of delta-matroids and isotropic systems which he had earlier introduced in~\cite{MR904585} and~\cite{MR919874}, respectively. They are defined on what he called a \emph{carrier}, a pair
$(U,\Omega)$ where $U$ is a finite set, called the \emph{ground set}, and $\Omega$ is a partition of $U$ into (non-empty) blocks called \emph{skew classes}. If each skew class
has size $q$, then $(U,\Omega)$ is a $q$-carrier. 

It is helpful to think of the skew classes, rather than the elements of $U$, as playing a role analogous to the elements of a matroid, and the elements of each skew class representing different `states' that the skew class might play. As an illustration, in Section~\ref{sec:rg} we consider multimatroids arising from ribbon graphs (or, equivalently, graphs in surfaces). There each edge of the ribbon graph corresponds to a skew class, 
and each skew class consists of three elements that can be thought of corresponding to an edge being either present, absent, or present but given a ``half-twist''.
We shall adopt notation that reflects this understanding.

\begin{notation}\label{not:1}
For each skew class $\sk{e}$, the members of the skew class are named by decorating the symbol $\sk{e}$. For example, for a skew class $\sk{e}$ we might write $\sk{e}=\{e_1, e_2,\ldots ,e_q\}$ or $\sk{e}=\{\dott{e}, \barr{e}, \hatt{e}\}$. (See Example~\ref{example1} for this convention in practice.) 
\expleend\end{notation}

For a carrier $(U,\Omega)$, a pair of distinct elements belonging to the same skew class is called a \emph{skew pair}. A \emph{subtransversal} of $\Omega$ is a subset of $U$
meeting each skew class at most once. 
The set of
subtransversals of $\Omega$ is denoted by $\mathcal S(\Omega)$. 
A \emph{transversal} $T$ of $\Omega$ is a subtransversal of $\Omega$ with $|T|=|\Omega|$, that is, a set containing precisely one element from each skew class. The set of
transversals of $\Omega$ is denoted by $\mathcal T(\Omega)$. 
We say that a subtransversal $S$ is a \emph{near-transversal} if $|S|=|\Omega|-1$, that is, it meets all but one skew class. 
The set of
near-transversals of $\Omega$ is denoted by $\mathcal N(\Omega)$. 

\medskip

A \emph{multimatroid} is a triple $Z=(U,\Omega,r_Z)$ so that $(U,\Omega)$ is a carrier and $r_Z$ is a \emph{rank function} $r_Z:\mathcal
S(\Omega)\rightarrow \ints_{\geq 0}$ satisfying the two axioms below. 
\begin{enumerate}[label = {(R\arabic*)}]
\item \label{r1} For each transversal $T$ of $\Omega$, the pair comprising $T$ and the restriction of $r_Z$ to subsets of $T$ forms a matroid.
\item \label{r2} For each subtransversal $S$ of $\Omega$ and skew pair $\{e_i,e_j\}$ from  a skew class disjoint from $S$, we have
    \[ r_Z(S\cup \{e_i\}) + r_Z(S\cup \{e_j\}) - 2r_Z(S) \geq 1.\]
\end{enumerate}
We generally omit the subscript $Z$ from $r_Z$ whenever the context is clear.  Axiom~(R1) implies the following: $r(\emptyset)=0$; for $A\in\mathcal S(\Omega)$ and  $e_i\in U$ such that $A\cup\{e_i\}\in\mathcal S(\Omega)$, $r(A)\leq r(A\cup\{e_i\})\leq r(A)+1$; and for $A$, $B\in\mathcal S(\Omega)$ such that $A\cup B\in\mathcal S(\Omega)$, $r(A\cup B)+r(A\cap B)\leq r(A)+r(B)$.

 It is convenient to extend the concepts we have defined for
carriers to multimatroids. For example, 
by a transversal of a multimatroid we mean a transversal of its underlying carrier, and we use $\mathcal{T}(Z)$ to denote the set of transversals of a multimatroid $Z$, and $\mathcal{N}(Z)$ the set of near-transversals. Furthermore, the ground set of a multimatroid is the ground set of its carrier. We sometimes use $U(Z)$ to denote the ground set of a multimatroid $Z$.
We say that a multimatroid is a \emph{$q$-matroid} if its carrier is a $q$-carrier (i.e., each skew class has exactly $q$ members). (The special case of $2$-matroids was introduced earlier than general multimatroids in~\cite{MR904585} under the name
\emph{symmetric matroids}.) 

An element $e_i$ of a multimatroid $(U,\Omega,r)$ is \emph{singular} if $r(\{e_i\})=0$. A skew class is \emph{singular} if it contains a singular element.
A multimatroid is \emph{non-degenerate} if every skew class has size at least two.   
If $S$ is a subtransversal of $Z$, then we define its \emph{nullity} $n_Z(S)$ by $n_Z(S)=|S|-r_Z(S)$, again omitting
the subscript when the context is clear.

Let $Z=(U,\Omega,r)$, let $e_i$ be in $U$, and suppose that $e_i$ is in the skew class $\sk{e}$. Then, following~\cite{MM2zbMATH01119073}, the \emph{elementary minor} of $Z$ by $e_i$, denoted by $Z \res e_i$, is the triple
$(U',\Omega',r')$ with $U'=U-e$,  $\Omega'= \Omega- \{\sk{e}\}$
and for all $S$ in $\mathcal S(\Omega')$, $r'(S)=r(S\cup\{e_i\})-r(\{e_i\})$.
It is straightforward to show that $Z \res e_i$ is a multimatroid.

An \emph{independent set} of $Z$ is a subtransversal $S$ with $n(S)=0$. Subtransversals that are not independent are \emph{dependent}. A \emph{basis} is a maximal independent subtransversal, and we let $\mathcal{B}(Z)$ denote the set of all bases of $Z$. 
 For any subtransversal $S$, we have $r(S)=\max_{B\in\mathcal B(Z)} |B\cap S|$, and consequently multimatroids are uniquely determined by their carrier and collection of bases. By \cite[Proposition~5.5]{MM1zbMATH01116184}, the bases of a non-degenerate multimatroid $Z$ are transversals of $Z$. 

\medskip

We are particularly interested in tight multimatroids. These were introduced by Bouchet in~\cite{MR1845490} and they include the 3-matroids of ribbon graphs (as discussed in Section~\ref{sec:rg}). 

\begin{definition}\label{def:tight}
A multimatroid $Z$ is \emph{tight} if it is non-degenerate and for every near-transversal $S$, there is a unique element, which we shall denote by $\ez{Z}{S}$,  in the skew
class disjoint from $S$ such that $r(S \cup \{\ez{Z}{S}\}) = r(S)$.
We omit the subscript $Z$, when the context is clear.
\end{definition}
Note that the uniqueness condition is not included in Bouchet's  definition of tight multimatroids, however uniqueness follows easily from~\ref{r2}. 
 Moreover, as shown in \cite[Proposition~4.1]{MR1845490}, if a multimatroid is tight, then so are all its elementary minors. 

An isomorphism between two multimatroids is a bijection respecting the partition into skew classes and preserving rank. If $Z_1$ and $Z_2$ are isomorphic, we write $Z_1\cong Z_2$. Notice that tightness is preserved under isomorphism and that every isomorphism maps singular elements to singular elements.

\begin{example}\label{example1}
 Let $(U,\{a,b,c\})$ be the 3-carrier with skew classes
 $\sk{a}=\{\dott a, \barr{a}, \hatt{a}\}$,
 $\sk{b}=\{\dott b, \barr{b}, \hatt{b}\}$, and
 $\sk{c}=\{\dott c, \barr{c}, \hatt{c}\}$. 
 We implicitly assume in this and all other examples that differently named elements are distinct.
 
 An example of a 3-matroid  $Z$ over this carrier has the following sixteen bases:
\begin{center}
\begin{tabular}{llll}
$\{ \dott {a}, \barr{b}, \barr{c}\}$,&
$\{ \hatt{a}, \dott {b}, \barr{c}\}$,&
$\{ \dott {a}, \hatt{b}, \dott {c}\}$,&
$\{ \hatt{a}, \hatt{b}, \dott {c}\}$,\\
$\{ \barr{a}, \dott {b}, \barr{c}\}$,&
$\{ \hatt{a}, \dott {b}, \dott {c}\}$,&
$\{ \dott {a}, \dott {b}, \hatt{c}\}$,&
$\{ \hatt{a}, \hatt{b}, \hatt{c}\}$,\\
$\{ \dott {a}, \dott {b}, \dott {c}\}$,&
$\{ \dott {a}, \hatt{b}, \barr{c}\}$,&
$\{ \dott {a}, \barr{b}, \hatt{c}\}$,&
$\{ \hatt{a}, \barr{b}, \hatt{c}\}$,\\
$\{ \hatt{a}, \barr{b}, \barr{c}\}$,&
$\{ \barr{a}, \hatt{b}, \barr{c}\}$,&
$\{ \barr{a}, \dott {b}, \hatt{c}\}$,&
$\{ \barr{a}, \hatt{b}, \hatt{c}\}$.
\end{tabular}
\end{center}

Of the remaining eleven transversals, all of which are dependent,
$\{ \barr{a}, \barr{b}, \dott {c}\}$
 has rank one and the rest have rank two.
This multimatroid can be shown to be tight, and, for example, $\e{\{\dott {a}, \barr{b}\}}=\dott{c}$.

Note that in our examples, the naming convention we have adopted for the elements of a skew class gives a partition of the ground set of the multimatroid into transversals determined by the type of decoration. We stress that, in general, the elements of a multimatroid do not have a natural partition into transversals. Nevertheless, as we shall see in Sections~\ref{sec:delta} and~\ref{sec:rg}, the elements of a multimatroid coming from a delta-matroid or from a ribbon graph may be naturally partitioned into transversals, and this motivates our choice of notation.\expleend
\end{example}

For use later, we give a second, similar example.
\begin{example}\label{example1b}
 Let $(U,\{a,f,g\})$ be the 3-carrier with skew classes
 $\sk{a}=\{\dott a, \barr{a}, \hatt{a}\}$, 
 $\sk{f}=\{\dott f, \barr{f}, \hatt{f}\}$, and
 $\sk{g}=\{\dott g, \barr{g}, \hatt{g}\}$. An example of a 3-matroid  $Z$ over this carrier has the following sixteen bases:
\begin{center}
\begin{tabular}{llll}
$\{ \dott{a}, \dott{f}, \dott{g}\}$,&
$\{ \dott{a}, \dott{f}, \hatt{g}\}$,&
$\{ \dott{a}, \barr{f}, \barr{g}\}$,&
$\{ \dott{a}, \barr{f}, \hatt{g}\}$,
\\
$\{ \dott{a}, \hatt{f}, \dott{g}\}$,&
$\{ \dott{a}, \hatt{f}, \barr{g}\}$,&
$\{ \barr{a}, \dott{f}, \barr{g}\}$,&
$\{ \barr{a}, \dott{f}, \hatt{g}\}$,
\\
$\{ \barr{a}, \barr{f}, \dott{g}\}$,&
$\{ \barr{a}, \barr{f}, \hatt{g}\}$,&
$\{ \barr{a}, \hatt{f}, \dott{g}\}$,&
$\{ \barr{a}, \hatt{f}, \barr{g}\}$,
\\
$\{ \hatt{a}, \dott{f}, \dott{g}\}$,&
$\{ \hatt{a}, \dott{f}, \barr{g}\}$,&
$\{ \hatt{a}, \barr{f}, \dott{g}\}$,&
$\{ \hatt{a}, \barr{f}, \barr{g}\}$.
\end{tabular}
\end{center}
\expleend
\end{example}

\section{Tensor products and two-sums}\label{sec:tensors}

The two-sum of multimatroids was introduced in~\cite{2sumpaper}. 
The idea behind it is as follows. 
We say that multimatroids $Z_1$
and $Z_2$ \emph{align at a non-singular skew class $\sk{e}$} if the intersection of their ground sets is $\sk{e}$, and $\sk{e}$ is a non-singular skew class of both $Z_1$ and $Z_2$.
Suppose  $Z_1$ and $Z_2$ are tight multimatroids that align at a non-singular skew class $\sk{e}$. Suppose also that $I_1$ and $I_2$ are independent near-transversals of $Z_1$ and $Z_2$, respectively, both avoiding $\sk{e}$. Then $I_1 \cup I_2$ is a basis of the two-sum $Z_1\oplus_{2} Z_2$ along $\sk{e}$ if and only if $\ez{Z_1}{I_1} \neq \ez{Z_2}{I_2}$. The idea of constructing a tensor product of $Z$ and $Z'$ is then to two-sum a copy of $Z'$ along each skew class of $Z$.

Formally the two-sum is defined as follows.
\begin{definition}\label{def:2sum}
Let $Z_1:=(U_1, \Omega_1, r_1)$ and $Z_2:=(U_2, \Omega_2, r_2)$ be tight multimatroids that align at a non-singular skew class $\sk{e}$. 
Then the multimatroid $Z= (U, \Omega, r)$ is the \emph{two-sum} along $\sk{e}$, denoted by $Z_1 \oplus_{2} Z_2$, if $U=(U_1\cup U_2)-\sk{e}$, $\Omega = (\Omega_1\cup \Omega_2)-\{\sk{e}\}$ 
and a transversal $B$ of $\Omega$ is a basis if and only if $B\cap U_1$ and $B\cap U_2$ are independent near-transversals of $Z_1$ and $Z_2$, respectively, with 
$\ez{Z_1}{B\cap U_1} \ne \ez{Z_2}{B\cap U_2}$.
\end{definition}
It is straightforward to show that $Z$ is genuinely a multimatroid~\cite{2sumpaper}.

\begin{example}
Suppose that $Z_1$ is the multimatroid in Example~\ref{example1}, and $Z_2$ that in Example~\ref{example1b}. We consider $Z_1 \oplus_2 Z_2$.
The set of independent near-transversals $T_1$ of $Z_1$ with 
$\ez{Z_1}{T_1}=\dott{a}$ is
$X_{\dott{a}} :=
\{
\{\dott{b},\barr{c}\}, 
\{\hatt{b}, \hatt{c}\}\}$;
where
$\ez{Z_1}{T_1}=\barr{a}$ is
$X_{\barr{a}} :=
\{
\{\dott{b},\dott{c}\}, 
\{\barr{b},\barr{c}\}, 
\{\barr{b},\hatt{c}\}
\{\hatt{b},\dott{c}\}\}$; and where
$\ez{Z_1}{T_1}=\hatt{a}$ is
$X_{\hatt{a}} :=
\{
\{\dott{b},\hatt{c}\}, 
\{\hatt{b},\barr{c}\}\}$.
The set of independent near-transversals $T_2$ of $Z_2$ with 
$\ez{Z_2}{T_2}=\dott{a}$ is
$Y_{\dott{a}} :=
\{
\{\dott{f},\barr{g}\}, 
\{\barr{f},\dott{g}\}\}$;
where $\ez{Z_2}{T_2}=\barr{a}$ is 
$Y_{\barr{a}} :=
\{
\{\dott{f},\dott{g}\}, 
\{\barr{f},\barr{g}\}\}$;
and where $\ez{Z_2}{T_2}=\hatt{a}$ is
$Y_{\hatt{a}} :=
\{
\{\dott{f},\hatt{g}\}, 
\{\barr{f},\hatt{g}\},
\{\hatt{f},\dott{g}\}, 
\{\hatt{f},\barr{g}\}\}$.
Then 
$Z_1 \oplus_2 Z_2$ has skew classes
$\sk{b}:=\{\dott b, \barr{b},\hatt{b}\}$,
$\sk{c}:=\{\dott c, \barr{c}, \hatt{c}\}$,
$\sk{f}:=\{\dott f, \barr{f}, \hatt{f}\}$, and
$\sk{g}:=\{\dott g, \barr{g}, \hatt{g}\}$.
 Its bases consist of all sets of the form $x\cup y$ where 
$x\in X_{\dott{a}}$ and $y\in Y_{\barr{a}} \cup Y_{\hatt{a}}$, 
or
$x\in X_{\barr{a}}$ and $y\in Y_{\dott{a}} \cup Y_{\hatt{a}}$,
or 
$x\in X_{\hatt{a}}$ and $y\in Y_{\dott{a}} \cup Y_{\barr{a}}$.\expleend
\end{example}

\medskip
The next two results are from~\cite{2sumpaper}. The proofs are straightforward.
\begin{proposition}\label{prop:twosumrank}
Let $Z_1:=(U_1, \Omega_1, r_1)$ and $Z_2:=(U_2, \Omega_2, r_2)$ be tight multimatroids that align at a non-singular skew class $\sk{e}$. 
Suppose that
$T_1$ is a near-transversal of $Z_1$ avoiding $\sk{e}$, and $T_2$ is a near-transversal of $Z_2$ avoiding $\sk{e}$. Then
\[
r_{Z_1 \oplus_2 Z_2}(T_1\cup T_2) 
=
\begin{cases} r_{Z_1}(T_1) +r_{Z_2}(T_2) 
&\text{if } 
\ez{Z_1}{T_1}\ne \ez{Z_2}{T_2}, \\
r_{Z_1}(T_1) +r_{Z_2}(T_2) -1
&\text{if } \ez{Z_1}{T_1}=\ez{Z_2}{T_2}. 
\end{cases}
\]
\end{proposition}

\begin{proposition}\label{lem:slack}
Let $Z_1:=(U_1, \Omega_1, r_1)$ and $Z_2:=(U_2, \Omega_2, r_2)$ be tight multimatroids that align at a non-singular skew class.  Then their two-sum  $Z_1 \oplus_{2} Z_2$ is also tight.
\end{proposition}

The two-sum operation is clearly commutative. The previous proposition suggests the possibility of forming the two-sum of two multimatroids and then forming a two-sum of the resulting multimatroid with a third multimatroid. The following associativity property also holds~\cite{2sumpaper}. Again its proof is straightforward.

\begin{proposition}\label{prop:sums}
Let $Z_1, Z_2, Z_3$ be tight multimatroids with the carrier of $Z_i$ being $(U_i,\Omega_i)$ for $i=1,2,3$. Suppose that $e$ and $f$ are distinct non-singular skew classes of $Z_2$, that $e$ is a non-singular skew class of $Z_1$, that $f$ is a non-singular skew class of $Z_3$, and that $U_1 \cap U_2 = e$, $U_2\cap U_3=f$ and  $U_1 \cap U_3 =\emptyset$. 
Then 
\[ \big( Z_1 \oplus_{2} Z_2\big)  \oplus_{2}   Z_3 
=
Z_1 \oplus_{2} \big( Z_2 \oplus_2 Z_3\big). 
\]
\end{proposition}

Proposition~\ref{prop:sums} allows us to carry out multiple two-sum operations in any order.
More precisely, we say a finite collection $\mathcal Z$ of tight multimatroids is \emph{graphable} if it 
satisfies both the following conditions.
\begin{itemize}
\item No element belongs to the ground set of three or more members of $\mathcal Z$.
\item For every two distinct members $Z_1$ and $Z_2$ of $\mathcal Z$, 
either the ground sets of $Z_1$ and $Z_2$ are disjoint or $Z_1$ and $Z_2$ align at a non-singular skew class.
\end{itemize}
If $\mathcal Z$ is graphable then we may form its \emph{graph} as follows. 
Its vertices are the members of $\mathcal Z$ and 
distinct vertices $Z_1$ and $Z_2$ are joined by an edge if the ground sets of $Z_1$ and $Z_2$ intersect.
 Whenever we consider the graph of a collection of multimatroids, we assume implicitly that the collection is graphable.
 Suppose that $G$, the graph of $\mathcal Z$, is a tree. Then Proposition~\ref{prop:sums} ensures that we can carry out all the two-sums corresponding to the edges of $G$ in any order without affecting the result. We denote the resulting multimatroid by $\bigoplus_2 \mathcal Z$.

We focus on a special case. Suppose that $Z_1$ is a tight multimatroid with carrier $(U_1,\Omega_1)$ in which every skew class is non-singular, and that $\mathcal Z=\{Z_e:e \in \Omega_1\}$ is a collection of tight multimatroids indexed by the skew classes of $Z_1$ 
such that, for each $e$, $Z_1$ and $Z_e$ align at the non-singular skew class $e$. 
Suppose that the graph of $\mathcal Z \cup\{Z_1\}$ is a star with central vertex $Z_1$. Then we say that $\bigoplus_2 (\mathcal Z \cup \{Z_1\})$ is the \emph{star product of} $Z_1$ and $\mathcal Z$. 

We will be particularly interested in the following case of the star product. Suppose that $Z_1$ and $\mathcal Z=\{Z_e:e \in \Omega_1\}$ are as described above. Now suppose that $Z_2$ is a tight multimatroid with a non-singular skew class $a$ so that for each $e \in \Omega_1$, the multimatroid $Z_e$ is isomorphic to $Z_2$ via an isomorphism $\phi_e$ with $\phi_e(e)=a$.
Finally, let $\Phi = \{\phi_e: e\in \Omega_1\}$. 
Then we say that the star product of $Z_1$ and $\mathcal Z$ is the \emph{tensor product} of $Z_1$ and $Z_2$ via $\Phi$, which we denote by $Z_1 \otimes_{\Phi} Z_2$.
 
\begin{remark}
The carrier of $Z_1 \otimes_{\Phi} Z_2$ is determined by $\Phi$. Moreover $Z_1 \otimes_{\Phi} Z_2$ is determined up to isomorphism by $Z_1$, $Z_2$ and the collection of maps $\{\phi_e\res e : e\in \Omega_1\}$, where $\phi_e\res e$ denotes the restriction of $\phi_e$ to the skew class $e$. 

But, as we shall see in the following example, two tensor products of $Z_1$ and $Z_2$ are not necessarily isomorphic, which is why we incorporate $\Phi$ into the notation.  
This ambiguity arises because the elements of a skew class of a multimatroid are generally neither ordered nor otherwise distinguished by, for example, a partition of the ground set of a multimatroid into transversals.
So, there is no canonical choice for a bijection between a skew class $e$ of $Z_1$ and the distinguished skew class $a$ of $Z_2$.

As we shall see in Theorems~\ref{thm:rgpdm} and~\ref{thm:rgtpf}, when we apply our results to the multimatroids coming from delta-matroids (including matroids) and ribbon graphs, respectively, there is a natural way to distinguish elements of each skew class and there are canonical bijections between each skew class of $Z_1$ and the skew class $a$ of $Z_2$.
\expleend\end{remark}

\medskip

Note that by repeated application of Proposition~\ref{lem:slack}, a tensor product of tight multimatroids is tight. We now give a small example of a tensor product of two 2-matroids.

\begin{example}\label{ex:smtens}
    Let $Z_1$ be the tight 2-matroid whose skew classes are $e_1:=\{\dott{e}_1,\barr{e}_1\}$ and $e_2 := \{\dott{e}_2,\barr{e}_2\}$, and whose bases are $\{\dott{e}_1,\barr{e}_{2}\}$ and $\{\barr{e}_1,\dott{e}_{2}\}$. 
    For $i=1,2$, let $Z_{e_i}$ be the tight \mbox{$2$-matroid} whose skew classes are $e_i:=\{\dott{e}_i,\barr{e}_i\}$, $b_i:=\{\dott{b}_i,\barr{b}_i\}$, and $c_i:=\{\dott{c}_i,\barr{c}_i\}$, and whose bases are $\{\dott{e}_i,\barr{b}_{i},\barr{c}_{i}\}$, $\{\barr{e}_i,\dott{b}_{i},\barr{c}_{i}\}$, and $\{\barr{e}_i,\barr{b}_{i},\dott{c}_{i}\}$. Now let $Z_2$ be the tight \mbox{$2$-matroid} whose skew classes are $a:=\{\dott{a}, \barr{a}\}$, $b:=\{\dott{b},\barr{b}\}$, and $c:=\{\dott{c},\barr{c}\}$
    and whose bases are
    $\{\dott{a},\barr{b},\barr{c}\}$, $\{\barr{a},\dott{b},\barr{c}\}$, and $\{\barr{a},\barr{b},\dott{c}\}$.  For $i=1,2$, let $\phi_{e_i}:U(Z_{e_i})\rightarrow U(Z_2)$ be given by 
    \[ \phi_{e_i}(\dott{e}_i):=\dott a,\
    \phi_{e_i}(\barr{e}_i):=\barr a,\
    \phi_{e_i}(\dott{b}_i):=\dott b,\
    \phi_{e_i}(\barr{b}_i):=\barr b,\
    \phi_{e_i}(\dott{c}_i):=\dott c,\
    \phi_{e_i}(\barr{c}_i):=\barr c.
    \] Then for $i=1,2$, $\phi_{e_i}$ is an isomorphism. Let $\Phi:=\{\phi_{e_1},\phi_{e_2}\}$.

    The tensor product $Z_1\otimes_{\Phi} Z_2$ 
    is the $2$-matroid with skew classes 
    $\{\dott{b}_1,\barr{b}_1\}$,
    $\{\dott{c}_1,\barr{c}_1\}$,
    $\{\dott{b}_2,\barr{b}_2\}$,
    $\{\dott{c}_2,\barr{c}_2\}$, 
and bases
 $\{\barr b_{1},\barr c_{1},\dott b_{2},\barr 
 c_{2}\}$,
    $\{\barr b_{1},\barr c_{1},\barr b_{2},\dott c_{2}\}$,
    $\{\dott b_{1},\barr c_{1},\barr b_{2},\barr c_{2}\}$, and
    $\{\barr b_{1},\dott c_{1},\barr b_{2},\barr c_{2}\}$.

Suppose we interchange the roles of $\dott e_2$ and $\barr e_2$ in $Z_{e_2}$, so that its bases become
$\{\barr{e}_2,\barr{b}_{2},\barr{c}_{2}\}$, $\{\dott{e}_2,\dott{b}_{2},\barr{c}_{2}\}$, $\{\dott{e}_2,\barr{b}_{2},\dott{c}_{2}\}$.
In order that $\phi_{e_2}$ is still an isomorphism between $Z_{e_2}$ and $Z_2$, we need to amend $\phi_{e_2}$ so that $\phi_{e_2}(\dott e_2):= \barr a$ and $\phi_{e_2}(\barr e_2):=\dott a$. The resulting tensor product $Z_1 \otimes_{\Phi} Z_2$ has bases
    $\{\barr b_{1},\barr c_{1},\barr b_{2},\barr c_{2}\}$,
    $\{\dott b_{1},\barr c_{1},\dott b_{2},\barr c_{2}\}$,
    $\{\dott b_{1},\barr c_{1},\barr b_{2},\dott c_{2}\}$,
    $\{\barr b_{1},\dott c_{1},\dott b_{2},\barr c_{2}\}$, and
    $\{\barr b_{1},\dott c_{1},\barr b_{2},\dott c_{2}\}$. As the number of bases has changed, we observe that different choices for $\Phi$ may yield non-isomorphic multimatroids.\expleend
\end{example}

\section{Related constructions}\label{sec:isotropic}
This section adds additional context and connects our work with analogous constructions on several related combinatorial structures. 
Our aim is to enable those with a passing familiarity with some of these structures to see the relationships between them, rather than give precise details.
We assume some background knowledge of delta-matroids and isotropic systems: this section may be skipped as the rest of the paper does not depend on its content.

The tensor product of graphs $G$ and $H$ (with a distinguished edge $e$) may be described in terms of two-sums. It is formed by taking the two-sum of $G$ with an isomorphic copy of $H$ along every edge of $G$, so that each two-sum is along an isomorphic copy of $e$.  
This formulation extends easily to matroids~\cite{Brylawski_2010} and using the notion of the two-sum for multimatroids from~\cite{2sumpaper} extends in an analogous way to delta-matroids and multimatroids.

As mentioned in the introduction, multimatroids were introduced by Bouchet~\cite{MM1zbMATH01116184}, partly in an attempt to simultaneously reconcile and generalize his earlier combinatorial structures, namely delta-matroids~\cite{MR904585} and isotropic systems~\cite{MR919874}. 
The definition of the two-sum construction for delta-matroids and multimatroids in~\cite{2sumpaper} was inspired by the matroid two-sum construction~\cite{MR1207587}. Indeed, the construction 
for even delta-matroids is exactly the same as that for matroids specified by their bases.
In this section we compare this construction with the very closely related composition operation for isotropic systems due to  Bouchet~\cite{zbMATH00015493} (but only later given the name composition~\cite{zbMATH01018470}) and more general work by Cunningham and Edmonds~\cite{zbMATH03689441} which developed a general theory of splits in combinatorial structures.

Binary matroids form an extremely important subclass of matroids. By appropriately modifying the definition of representability, one can define the class of binary delta-matroids~\cite{zbMATH04162893,zbMATH00008503}.
Every matroid (when specified by its bases) can be thought of as a delta-matroid and the concepts of representability are consistent in the sense that a matroid is binary if and only if it is binary when considered as a delta-matroid. A consequence of the definition of representability is that a normal binary delta-matroid, that is, one in which the empty set is feasible,
is determined by its feasible sets of size at most two. Thus (as with binary matroids) every binary delta-matroid is determined by specifying one feasible set $F$ and the collection of those feasible sets $F'$ with $|F\bigtriangleup F'|\leq 2$. So, up to twisting, a binary delta-matroid may be described by a graph, known as its \emph{fundamental graph}. Note that this graph describes `feasible set exchanges' and is not related to the delta-matroid in the way that the cycle matroid of a graph  is related to that graph.

If, rather than starting with a binary matroid defined in terms of its bases, instead one starts with its circuit space, and seeks to generalize this concept, then one is led in a highly non-trivial way to Bouchet's isotropic systems~\cite{MR919874}, which, very roughly, incorporate both the circuit and cocircuit spaces of a binary matroid in a subtle way. Bouchet proved that there is a collection of graphs~\cite{zbMATH04081574} associated with each isotropic system, with two graphs belonging to the same such collection if and only if they are locally equivalent. 
(Two graphs are \emph{locally equivalent} if one can be obtained from the other by a sequence of operations each of which involves 
complementing the edges between all neighbours of some vertex.) Up to isomorphism, an isotropic system is determined by any one of these graphs. 
The equivalence of binary delta-matroids and isotropic systems follows from this. More recent combinatorial constructions due to 
Traldi~\cite{zbMATH06381893} and Brijder and Traldi~\cite{MR3577648}, namely isotropic matroids and tight binary $3$-matroids, respectively, are also equivalent to isotropic systems and binary delta-matroids. 
The relationships between these different constructions are discussed in~\cite{MR3577648}.

Cunningham and Edmonds~\cite{zbMATH03689441} construct a general theory of what they call \emph{splits}. Roughly speaking, this involves decomposing a combinatorial structure into two smaller structures with a minimal but non-trivial intersection. They show that providing the collection of possible splits obeys various properties shared by two-sums in graphs, then one can neatly characterize all the possible decompositions obtained by successively forming a split. They give specific specializations to graphs and the collection of circuits of a matroid coming from two-sums. Cunningham~\cite{zbMATH03783038} describes a \emph{join} operation on pairs of graphs or digraphs having a common vertex and considers the corresponding splits. This falls within the general framework of~\cite{zbMATH03689441}.

In~\cite{zbMATH00015493}, Bouchet introduced the connectivity function of an isotropic system, which is analogous to the connectivity function of a matroid~\cite{MR1207587}. Partitions of the ground set of an isotropic system $\mathcal I$ with connectivity one correspond to the splits
from~\cite{zbMATH03783038} in each fundamental graph of $\mathcal I$, something which Bouchet exploited to prove an analogue of Tutte's `Wheels and Whirls Theorem' for isotropic systems. Bouchet defines the \emph{composition} $I$ of two isotropic systems $I_1$ and $I_2$ to be such that any fundamental graph of $I$ is the join of the corresponding fundamental graphs of $I_1$ and $I_2$.  
The composition operation was developed further in~\cite{zbMATH01018470}.

By considering the effect of the composition operation on the collection of feasible sets of the associated normal binary delta-matroids, one obtains a notion of \emph{composition} for delta-matroids~\cite{zbMATH00750738}. For normal binary delta-matroids, the consistency of the definitions of composition in terms of fundamental graphs and 
feasible sets is explained by Lemma~4.4 of \cite{geelen-thesis}. The identity proved there also follows from the proof of Proposition~2.3 of~\cite{zbMATH01355177}. However, the notion of composition
does not extend to multimatroids in a canonical way. 
Moreover it is not consistent with the existing notion of the two-sum in ribbon graphs~\cite{Huggett_2011,maya2}, nor with that in matroids. Indeed, the composition of two matroids (considered as delta-matroids) would not be a matroid. Thus the definition of the two-sum for delta-matroids given in Section~\ref{sec:delta} is not exactly Cunningham and Bouchet's composition operation, although for even delta-matroids the difference is only slight: it follows immediately from the definitions in terms of feasible sets that if $D_1$ and $D_2$ are two even delta-matroids with a common element $e$, then the composition of  $D_1$ and $D_2$ is $D_1 \oplus_2 (D_2 \ast \{e\})$, where $\ast$ denotes the twist operation discussed in Section~\ref{sec:delta}.

Traldi has obtained an expression for the weighted interlace polynomial of the composition of two graphs~\cite{zbMATH05700345}. The weighted interlace polynomial of a graph $G$ coincides with the weighted transition polynomial of the normal binary delta-matroid described by $G$. 
Thus Theorem~\ref{thm:vfsafe2sum} generalizes Traldi's result with the proviso that the definitions of composition and $2$-sum do not exactly coincide, as we have explained in this section.

\section{The transition polynomial}\label{sec:results}
This section contains our main results: expressions for the transition polynomial of a two-sum and a star product, and an extension of Brylawski's tensor product formula to tight multimatroids. 
The \emph{weighted transition polynomial} of a multimatroid $Z=(U,\Omega,r)$, introduced in~\cite{MR3191496}, is
\[ Q(Z;\{x_u\}_{u\in U}, t):= \sum_{T\in \mathcal T(Z)} t^{n(T)}\vect x_T ,\]
where $\vect x_T=\prod_{u\in T}  x_u$. We often abbreviate $\{x_u\}_{u \in U}$ to $\vect x$, when $U$ is clear from the context, and sometimes just write $Q(Z)$ for $Q(Z;\vect x,t) = Q(Z;\{x_u\}_{u\in U}, t)$.

Notice that if $U=\emptyset$, then $\mathcal T(\Omega)=\{\emptyset\}$, so $Q(Z;\vect x,t)=1$.
Here we treat  the 
$x_u$
as formal variables, giving a polynomial in 
$\mathbb{Z}[t, \{x_u\}_{u\in U}]$.
However they can be, and often are, treated as parameters.

\begin{example}\label{ex:q1}
Let $Z$ be the multimatroid in Example~\ref{example1}. Then the transition polynomial $Q(Z;\vect x, t)$ has 27 summands. Sixteen of these come from the bases and give the $t^0$ terms, the transversal $\{\barr{a},\barr{b},\dott{c}\}$ gives the only $t^2$ term, and the remaining eleven transversals give the $t^1$ terms. So the transition polynomial takes the form
\[
(x_{\dott{a}}x_{\barr{b}}x_{\barr{c}}
+x_{\barr{a}}x_{\dott{b}}x_{\barr{c}}
+\cdots +
x_{\barr{a}}x_{\hatt{b}}x_{\hatt{c}})
+
t(
x_{\dott{a}}x_{\dott{b}}x_{\barr{c}}+
\cdots +
x_{\barr{a}}x_{\barr{b}}x_{\hatt{c}}
)
+
t^2(x_{\barr{a}}x_{\barr{b}}x_{\dott{c}})
.\]
In practice, it is common to consider some substitution or specialisation of the variables $x_T$. For example, taking, for each $\sk{e}$, $x_{\dott{e}}=u$, $x_{\barr{e}}=v$, $x_{\hatt{e}}=0$ gives a polynomial
\[Q(Z;(u,v,0), t) =  
(u^3+2uv^2)+t(3u^2v+v^3) +t^2uv^2,
  \]
that is related to the Tutte polynomial of a delta-matroid or ribbon graph, as described in Sections~\ref{sec:delta} and~\ref{sec:rg}, respectively.
\expleend
\end{example}
To describe a substitution or specialisation of the variables we often abuse notation slightly and write expressions such as 
$Q(Z; \{x_u=y_u\}_{u \in U},t)$ meaning that each variable $x_u$ is set to $y_u$ or
$Q(Z; \vect x,t\mid  x_{e_i} = y_{e_i}, i=1,\ldots,q)$ by which we mean that for $i=1,\ldots, q$, the variable $x_{e_i}$ should be set to $y_{e_i}$. All the other variables are unaffected. When we do this, $y_{e_i}$ may be a polynomial in $t$ and variables indexed by elements of another multimatroid. Clearly such a substitution changes the polynomial ring to which $Q(Z)$ belongs.

\medskip

We first give an expression for the transition polynomial of a two-sum, $Z_1\oplus_2 Z_2$, showing that it may be obtained from $Q(Z_1)$ by setting the variables 
indexed by the elements of the skew class along which the two-sum is formed to appropriate polynomials in $t$ and variables indexed by some of the elements of $Z_2$.

We need the following result which is one case of Theorem~6 from~\cite{MR3191496}. 

\begin{proposition}[{\cite[Theorem 6]{MR3191496}}]\label{prop:delcon}
Let $Z:=(U,\Omega,r)$ be a multimatroid and let $e$ be a non-singular skew class of $\Omega$. We have
\[
Q(Z;\vect x,t) =
\sum_{u\in e} x_u Q(Z\res u;\vect x', t),
\]
where $\vect x'$ is obtained from $\vect x$ by removing the entries indexed by the elements of $e$.
\end{proposition}

The two-sum formula is given by the following theorem.
\begin{theorem}\label{thm:trans2sum}
Let $Z_1$ and $Z_2$ be tight multimatroids that align at a non-singular skew class $e:=\{e_1,\ldots,e_q\}$. Let 
$U:=U(Z_1 \oplus_2 Z_2)$
and
$U'_2:=U(Z_2)-e$.
Then
\[
Q(Z_1\oplus_2 Z_2; \{x_u\}_{u \in U},t) =
Q(Z_1; \{x_u\}_{u \in U(Z_1)},t \mid x_{e_i} = y_{e_i}, i=1,\ldots,q),
\]
where $\{y_u\}_{u\in e}$ is the unique solution to the following system of linear equations in the ring $\mathbb{Z}[t, \{x_u\}_{u\in U'_2}]$.  
\begin{samepage}\begin{align*}
        Q(Z_2\res e_1;\{x_u\}_{u \in U'_2},t)&=ty_{e_1}+y_{e_2}+\dots +y_{e_q}, \\
        Q(Z_2\res e_2;\{x_u\}_{u \in U'_2},t)&=y_{e_1}+ty_{e_2}+\dots +y_{e_q}, \\
        & \vdots \\
        Q(Z_2\res e_q;\{x_u\}_{u \in U'_2},t)&=y_{e_1}+y_{e_2}+\dots +ty_{e_q}.    \end{align*}\end{samepage}
\end{theorem}

\begin{proof}
Let $\mathcal {N}_1$ and $\mathcal{N}_2$ denote the sets of near-transversals missing $e$ of $Z_1$ and $Z_2$, respectively. For $i=1,\ldots,q$, let
\[ y_{e_i} =\sum_{\substack{S_2\in \mathcal N_2\\ \theta(S_2)=e_i}} \vect x_{S_2} t^{n(S_2)} .\]
Thus 
\[ Q(Z_2;\{x_u\}_{u \in U(Z_2)},t) = \sum_{i=1}^q y_{e_i} (x_{e_1} + \cdots  + x_{e_{i-1}} + tx_{e_i} + x_{e_{i+1}} + \cdots + x_{e_q}).\]
By considering the monomials of $Q(Z_2)$ containing $x_{e_i}$ and using Proposition~\ref{prop:delcon}, we deduce that for $i=1,\ldots,q$,
\[ Q(Z_2\res e_i; \{x_u\}_{u \in U'_2},t) = y_{e_1} + \cdots + y_{e_{i-1}} + ty_{e_i} + y_{e_{i+1}} + \cdots + y_{e_q}.\]
This gives a system of $q$ simultaneous equations in
$\mathbb{Z}[t, \{x_u\}_{u\in U'_2}]$ for $y_{e_1}$, \ldots, $y_{e_q}$. By construction, these have a solution and because the matrix of the coefficients of the simultaneous equations has non-vanishing determinant, the solution is unique.
Therefore $y_{e_1}$, \ldots, $y_{e_q}$ may be recovered from $Q(Z_2\res e_1)$, \ldots, $Q(Z_2\res e_q)$, (and consequently from $Q(Z_2)$).

We have
\begin{align*}
Q(Z_1 \oplus_2 Z_2) &= \sum_{S_1\in \mathcal N_1} \sum_{S_2\in \mathcal N_2} \vect x_{S_1} \vect x_{S_2} t^{n(S_1 \cup S_2)} \\
&= \sum_{S_1\in \mathcal N_1} \sum_{i=1}^q \sum_{\substack{S_2\in \mathcal N_2 \\ \theta(S_2)=e_i}} \vect x_{S_1} \vect x_{S_2} t^{n(S_1 \cup S_2)}\\
&= \sum_{S_1\in \mathcal N_1} \vect x_{S_1} t^{n(S_1)} \Big(\sum_{\substack{1 \leq i \leq q \\ e_i \ne \theta(S_1)}} \sum_{\substack{S_2\in \mathcal N_2\\ \theta(S_2)=e_i}} \vect x_{S_2} t^{n(S_2)} + 
\sum_{\substack{ S_2\in \mathcal N_2\\ \theta(S_2)=\theta(S_1)}}   \vect x_{S_2} t^{n(S_2)+1}\Big)\\
&= \sum_{S_1\in \mathcal N_1} \vect x_{S_1} t^{n(S_1)} \Big(\sum_{\substack {1 \leq i \leq q \\ e_i \ne \theta(S_1)}} y_{e_i} +  y_{\theta(S_1)} t\Big)\\
&= Q( Z_1;\{x_u\}_{u \in U(Z_1)},t\mid  x_{e_i} = y_{e_i}, i=1,\ldots,q ),
\end{align*}
as required.
\end{proof}

\medskip
We now use Theorem~\ref{thm:trans2sum} to obtain an expression for the star product of tight multimatroids.

\begin{theorem}\label{thm:star}
Suppose that $Z_1$ is a tight multimatroid with carrier $(U_1,\Omega_1)$. Assume that each skew class $e$ of $Z_1$ is non-singular and has the form $e=\{e_1,\ldots,e_{q_e}\}$, where $q_e=|e|$. 
Now suppose that $\mathcal Z=\{Z_e:e \in \Omega_1\}$ is a collection of tight multimatroids indexed by the skew classes of $Z_1$ such that, for each $\sk{e}$, $Z_1$ and $Z_e$ align at the non-singular skew class $\sk{e}$, and that the graph of $\mathcal Z \cup\{Z_1\}$ is a star with 
central vertex $Z_1$.
Let $Z'$ denote the star product of $Z_1$ and $\mathcal Z$ and let $U:=U(Z')$. Then  
        \[ Q(Z'; \{x_u\}_{u \in U},t )=Q(Z_1;\{x_{e_{i}}=y_{e_i}\}_{e \in \Omega_1,\, 1 \leq i \leq q_e}),\]
        where for each $e \in \Omega_1$, $\{y_u\}_{u\in e}$ 
        is the unique solution to the following system of linear equations in the ring $\mathbb{Z}[t, \{x_u\}_{u\in U(Z_e)-e}]$.  
\begin{align*}
        Q(Z_e\res e_{1};\{x_u\}_{u\in U(Z_e)-e},t)&=ty_{e_1}+y_{e_2}+\dots +y_{e_{q_e}}, \\
        Q(Z_e\res e_{2};\{x_u\}_{u\in U(Z_e)-e},t)&=y_{e_1}+ty_{e_2}+\dots +y_{e_{q_e}}, \\
        & \vdots \\
        Q(Z_e\res e_{q_e};\{x_u\}_{u\in U(Z_e)-e},t)&=y_{e_1}+y_{e_2}+\dots +ty_{e_{q_e}}.
    \end{align*}
\end{theorem}

\begin{proof}
    The theorem follows from Theorem~\ref{thm:trans2sum} and Proposition~\ref{prop:sums}. Forming the star product of $Z_1$ and $\mathcal Z$ requires $|\Omega_1|$ two-sums, which by Proposition~\ref{prop:sums} may be performed in any order. For $k=0,\ldots,|\Omega_1|$, let $Z^k$ denote the multimatroid formed after $k$ of these two-sums. Then $Z^0=Z_1$, $Z^{|\Omega_1|}=Z'$ and for $k=1,\ldots,|\Omega_1|$, $Z^k=Z^{k-1}\oplus_2 Z_e$ for some $e$ in $\Omega_1$. 
Fix $k$ and suppose that $Z^k=Z^{k-1}\oplus_2 Z_e$. 
By Theorem~\ref{thm:trans2sum}, $Q(Z^k)$ may be obtained from $Q(Z^{k-1})$ by making an appropriate substitution, depending on $Q(Z_e)$, for the variables of $Q(Z^{k-1})$ indexed by elements of $e$ and these new variables are unchanged thereafter. Moreover, by Theorem~\ref{thm:trans2sum}, the correct substitution for $x_{e_i}$, $i=1,\ldots,q_e$, is the one given in the theorem statement. The result follows by an easy induction on $k$.
\end{proof}

Using Theorem~\ref{thm:star} we obtain an expression for the tensor product of two multimatroids. Suppose that $\phi$ is an isomorphism between multimatroids $Z_1$ and $Z_2$. It is useful to abuse notation and also use $\phi: \ints[t, \{x_u\}_{u \in U(Z_1)} ]\rightarrow \ints[t, \{x_u\}_{u \in U(Z_2)} ]$ for the ring isomorphism with $\phi(x_u)=x_{\phi(u)}$ and $\phi(t)=t$.

\begin{theorem}\label{thm:main}
Let $q\geq 2$ and suppose that $Z_1$ is a tight $q$-matroid with set $\Omega_1$ of skew classes. Assume that each skew class $e$ in $\Omega_1$ is non-singular and has the form $e=\{e_1,\ldots,e_{q}\}$.
Now suppose that $\mathcal Z=\{Z_e:e \in \Omega_1\}$ is a collection of tight multimatroids indexed by $\Omega_1$ having pairwise disjoint ground sets so that for each $e$ in $\Omega_1$, $Z_1$ and $Z_e$ align at the non-singular skew class $e$.
Finally suppose that $Z_2$ is a tight multimatroid with a non-singular skew class $a:=\{a_1,\ldots,a_q\}$ so that for each $e \in \Omega_1$, the multimatroid $Z_e$ is isomorphic to $Z_2$ via an isomorphism $\phi_e$ with $\phi_e(e)=a$, let $\Phi := \{\phi_e: e\in \Omega_1\}$ and let $U'_2=U(Z_2)-a$. Then
\[ Q(Z_1 \otimes_{\Phi} Z_2)
= Q(Z_1;
\{x_{e_{i}}=\phi_e^{-1}(y_{\phi_e(e_i)})\}_{e \in \Omega_1, 1\leq i \leq q},t),\]
where
$\{y_u\}_{u\in a}$
        is the unique solution to the following system of linear equations in the ring $\mathbb{Z}[t, \{x_u\}_{u\in U'_2}]$.  
\begin{align*}
        Q(Z_2\res a_{1};
        \{x_u\}_{u\in U'_2},t)&=ty_{a_1}+y_{a_2}+\dots +y_{a_{q}}, \\
        Q(Z_2\res a_{2};
        \{x_u\}_{u\in U'_2},t)&=y_{a_1}+ty_{a_2}+\dots +y_{a_{q}}, \\
        & \vdots \\
        Q(Z_2\res a_{q};\{x_u\}_{u\in U'_2},t)&=y_{a_1}+y_{a_2}+\dots +ty_{a_{q}}.
    \end{align*}
\end{theorem}
\begin{proof}
First note that the conditions on the carriers of $Z_1$ and the multimatroids in $\mathcal Z$ ensure that the graph of $\{Z_1\} \cup \mathcal Z$ is a star with 
central vertex $Z_1$.
 Thus the star product of $Z_1$ and $\mathcal Z$ is well defined and is isomorphic
 to $Z_1 \otimes_{\Phi} Z_2$. 

Using Theorem~\ref{thm:star}, it is enough to show that for each $e$ in $\Omega_1$ and
$i=1,\ldots,q$,
we have $\phi_e^{-1}(y_{\phi_e(e_i)})=y_{e_i}$, where $\{y_u\}_{u\in e}$ is the unique solution to the following system of linear equations.
\begin{align*}
        Q(Z_e\res e_{1};\vect x,t)&=ty_{e_1}+y_{e_2}+\dots +y_{e_{q}}, \\
        Q(Z_e\res e_{2};\vect x,t)&=y_{e_1}+ty_{e_2}+\dots +y_{e_{q}}, \\
        & \vdots \\
        Q(Z_e\res e_{q};\vect x,t)&=y_{e_1}+y_{e_2}+\dots +ty_{e_{q}}.
    \end{align*}
Fix $e$ in $\Omega$. Then, as $\phi_e$ is an isomorphism, we need to show that for $i=1,\ldots,q$, we have
$y_{\phi_e(e_i)} = \phi_e(y_{e_i})$.
For $i=1,\ldots,q$ we have
\[ Q(Z_2\res  \phi_e(e_i)) = y_{\phi_e(e_1)} + \cdots + y_{\phi_e(e_{i-1})} + ty_{\phi_e(e_i)} + y_{\phi_e(e_{i+1})} + \cdots + y_{\phi_e(e_q)}. \]
On the other hand,
\begin{align*}
\MoveEqLeft{\phi_e(y_{e_1}) + \cdots + 
\phi_e(y_{e_{i-1}}) +
t\phi_e(y_{e_i}) +
\phi_e(y_{e_{i+1}}) + \cdots + \phi_e(y_{e_q})} \\
&=\phi_e(y_{e_1} + \cdots + y_{e_{i-1}}
+ ty_{e_{i}} + y_{e_{i+1}} + \cdots + y_{e_q}) \\&= \phi_e(Q(Z_e\res e_i)).\end{align*}
As the restriction of $\phi_e$ to $U(Z_2)-e$ is an isomorphism from $Z_e|e_i$ to $Z_2|\phi(e_i)$, we have $\phi_e(Q(Z_e\res e_i))= Q(Z_2\res  \phi_e(e_i))$.
Thus 
$(y_{\phi_e(e_1)}, \ldots,y_{\phi_e(e_q)} )$ and  $(\phi_e(y_{e_1}), \ldots,\phi_e(y_{e_q}) )$ satisfy the same set of simultaneous equations which are known to have a unique solution by Theorem~\ref{thm:star}.
This completes the proof.
\end{proof}

\begin{example}
Let $Z_1$, $Z_2$ and $\Phi$ be as described at the beginning of Example~\ref{ex:smtens}. Then
\begin{align*}
Q(Z_1;\vect x,t) &=
tx_{\dott e_{1}}x_{\dott e_{2}}+x_{\dott e_{1}}x_{\barr e_{2}}+x_{\barr e_{1}}x_{\dott e_{2}}+tx_{\barr e_{1}}x_{\barr e_{2}},
\\
Q(Z_2\res \dott a;\vect x,t) &=
t^2x_{\dott b}x_{\dott c}+tx_{\dott b}x_{\barr c}+tx_{\barr b}x_{\dott c}+x_{\barr b}x_{\barr c},
\\
Q(Z_2\res \barr a;\vect x,t) &=
tx_{\dott b}x_{\dott c}+x_{\dott b} x_{\barr c}+x_{\barr b}x_{\dott c}+tx_{\barr b}x_{\barr c}.
\end{align*}
We obtain $y_{\dott a}$ and $y_{\barr a}$ as the unique solution to
\begin{align*}
    t^2x_{\dott b}x_{\dott c}+tx_{\dott b}x_{\barr c}+tx_{\barr b}x_{\dott c}+x_{\barr b}x_{\barr c} &= t y_{\dott a} + y_{\barr a},
    \\
    tx_{\dott b}x_{\dott c}+x_{\dott b} x_{\barr c}+x_{\barr b}x_{\dott c}+tx_{\barr b}x_{\barr c} &= y_{\dott a} + t y_{\barr a}.
\end{align*}

Thus 
\begin{align*}
y_{\dott a}&=tx_{\dott b}x_{\dott c}+x_{\dott b}x_{\barr c}+x_{\barr b}x_{\dott c},
\\
y_{\barr a}&=x_{\barr b}x_{\barr c}.
\end{align*}

For $i=1,2$ we have $\phi_{e_i}(\dott e_i)=\dott a$ and $\phi_{e_i}(\barr e_i)=\barr a$. So the substitution 
from Theorem~\ref{thm:main}
is given by
\begin{align*}
    x_{\dott e_1} &=  \phi^{-1}_{e_1} (y_{\dott a}) = tx_{\dott b_1}x_{\dott c_1}+x_{\dott b_1}x_{\barr c_1}+x_{\barr b_1}x_{\dott c_1},\\
    x_{\barr e_1} &= \phi^{-1}_{e_1} (y_{\barr a}) = x_{\barr b_1}x_{\barr c_1}, \\
    x_{\dott e_2} &= \phi^{-1}_{e_2} (y_{\dott a})= tx_{\dott b_2}x_{\dott c_2}+x_{\dott b_2}x_{\barr c_2}+x_{\barr b_2}x_{\dott c_2},\\
    x_{\barr e_2} &= \phi^{-1}_{e_2} (y_{\barr a})=x_{\barr b_2}x_{\barr c_2}.
\end{align*}

So we get 
\begin{align*} \MoveEqLeft{Q(Z_1\otimes_{\Phi} Z_2)}\\
&= 
t(tx_{\dott b_{1}}x_{\dott c_{1}}+x_{\dott b_{1}}x_{\barr c_{1}}+x_{\barr b_{1}}x_{\dott c_{1}})    
(tx_{\dott b_{2}}x_{\dott c_{2}}+x_{\dott b_{2}}x_{\barr c_{2}}+x_{\barr b_{2}}x_{\dott c_{2}})\\
& \phantom{=} {} +(tx_{\dott b_{1}}x_{\dott c_{1}}+x_{\dott b_{1}}x_{\barr c_{1}}+x_{\barr b_{1}}x_{\dott c_{1}})    x_{\barr b_{2}}x_{\barr c_{2}}
+x_{\barr b_{1}}x_{\barr c_{1}}   
(tx_{\dott b_{2}}x_{\dott c_{2}}+x_{\dott b_{2}}x_{\barr c_{2}}+x_{\barr b_{2}}x_{\dott c_{2}})\\
& \phantom{=} {}
+tx_{\barr b_{1}}x_{\barr c_{1}}x_{\barr b_{2}}x_{\barr c_{2}}\\
&= t^3x_{\dott b_1}x_{\dott c_1}x_{\dott b_2}x_{\dott c_2} 
+ t^2x_{\dott b_1}x_{\dott c_1}x_{\dott b_2}x_{\barr c_2}
+ t^2x_{\dott b_1}x_{\dott c_1}x_{\barr b_2}x_{\dott c_2}
+ tx_{\dott b_1}x_{\dott c_1}x_{\barr b_2}x_{\barr c_2}\\
&\phantom{=} {}+ t^2x_{\dott b_1}x_{\barr c_1}x_{\dott b_2}x_{\dott c_2}
+ tx_{\dott b_1}x_{\barr c_1}x_{\dott b_2}x_{\barr c_2}
+ tx_{\dott b_1}x_{\barr c_1}x_{\barr b_2}x_{\dott c_2}
+ x_{\dott b_1}x_{\barr c_1}x_{\barr b_2}x_{\barr c_2}\\
&\phantom{=} {}+ t^2x_{\barr b_1}x_{\dott c_1}x_{\dott b_2}x_{\dott c_2}
+ tx_{\barr b_1}x_{\dott c_1}x_{\dott b_2}x_{\barr c_2}
+ tx_{\barr b_1}x_{\dott c_1}x_{\barr b_2}x_{\dott c_2}
+ x_{\barr b_1}x_{\dott c_1}x_{\barr b_2}x_{\barr c_2}\\
&\phantom{=} {}+ tx_{\barr b_1}x_{\barr c_1}x_{\dott b_2}x_{\dott c_2}
+ x_{\barr b_1}x_{\barr c_1}x_{\dott b_2}x_{\barr c_2}
+ x_{\barr b_1}x_{\barr c_1}x_{\barr b_2}x_{\dott c_2}
+ tx_{\barr b_1}x_{\barr c_1}x_{\barr b_2}x_{\barr c_2}.
\end{align*}
\end{example}

\section{Delta-matroids}\label{sec:delta}

In this section we exploit connections between multimatroids and delta-matroids to obtain Brylawski-type formulas for the tensor product of delta-matroids. Theorem~\ref{thm:rgpdm} gives the tensor product formula for the weighted transition polynomial for vf-safe delta-matroids; Theorem~\ref{thm:dmeven} gives the corresponding formula for even delta-matroids.  We conclude the section by rewriting the latter formula in terms of Tutte polynomials of delta-matroids, deriving Brylawski's original formula for matroids as a consequence. 

Delta-matroids and equivalent structures were introduced in the mid 1980s by several authors~\cite{MR904585,zbMATH04070920,zbMATH03985246}, but principally studied by Bouchet.
A \emph{delta-matroid} is a pair $(E,\mathcal{F})$ where $E$ is a finite set called the \emph{ground set} and $\mathcal{F}$ is a non-empty collection of subsets of $E$ satisfying the \emph{symmetric exchange axiom}:
for all triples $(X,Y,u)$ with $X$ and $Y$ in $\mathcal{F}$ and $u\in X\bigtriangleup Y$ (where $\bigtriangleup$ denotes symmetric difference of sets), there is an element $v\in X\bigtriangleup Y$ (perhaps $u$ itself) such that $X\bigtriangleup \{u,v\}$ is in $\mathcal{F}$.
The elements of $\mathcal{F}$ are called \emph{feasible sets}. If $D$ is a delta-matroid with ground set $E$, we sometimes say that $D$ is a delta-matroid on $E$.
A \emph{coloop} of a delta-matroid is an element belonging to every feasible set; a \emph{loop} is an element belonging to no feasible set. Delta-matroids $D_1:=(E_1,\mathcal F_1)$ and $D_2:=(E_2,\mathcal F_2)$ are \emph{isomorphic} if there is a bijection $\mu:E_1\rightarrow E_2$ that preserves feasible sets.

Let $D:=(E,\mathcal F)$ be a delta-matroid with element $e$. Suppose first that $e$ is not a coloop of $D$. Then we define $D\ba e$, the \emph{deletion} of $e$, to be the pair 
\[ (E-\{e\},\{ F \in \mathcal {F} : e\notin F\}).\]
Now suppose that $e$ is not a loop of $D$. Then we define  
$D/ e$, the \emph{contraction} of $e$, to be the pair 
\[ (E-\{e\},\{ F-\{e\} : F \in \mathcal {F} \text{ and } e \in F\}).\]
If $e$ is either a coloop or a loop of $D$, then one of $D\ba e$ and $D/e$ is defined. In this case, we define whichever of $D\ba e$ and $D/e$ is so far undefined by setting $D\ba e = D/e$. It is easy to check that both $D\ba e$ and $D/e$ are delta-matroids. 

Let $D:=(E,\mathcal F)$ be a delta-matroid.
For a subset $A$ of $E$, the \emph{twist} of $D$ by $A$, denoted by $D*A$, is given by $(E,\{X\bigtriangleup A: X\in\mathcal{F}\})$. It is straightforward to show that $D*A$ is indeed a delta-matroid.

Given an element $e$ in $E$, 
following~\cite{zbMATH05982480}, we define the \emph{loop complement} of $D$ by~$e$, which we denote by $D+e$, to be the pair $(E,\mathcal{F}')$, where  \[\mathcal{F}'=:\mathcal{F}\bigtriangleup\{X\cup \{e\} : X\in\mathcal{F}, e\notin X\}.\] Thus a set $A$ not containing $e$ is in $\mathcal F'$ if and only it is in $\mathcal F$, and a set $A$ containing $e$
is in $\mathcal F'$ if and only if precisely one of $A$ and $A-\{e\}$ is in $\mathcal F$.
If $e_1, e_2\in E$, then $(D+e_1)+e_2=(D+e_2)+e_1$, and so for a subset $A$ of $E$ we can unambiguously
define the loop complement $D+A$ of $D$ by $A$, by forming the loop complement with respect to each element of $A$ in turn.
The set of delta-matroids is not closed under loop complementation. A delta-matroid is said to be \emph{vf-safe} if the application of every sequence of twists and loop complements results in a delta-matroid. Any delta-matroid that may be obtained from a delta-matroid $D$ by a sequence of twists and loop complementations is said to be a $\emph{twisted dual}$ of $D$.

It was been shown in~\cite{zbMATH05982480} that the operations of twisting ${}\ast e$ and loop complementation ${}+e$ with respect to an element $e$ are involutions and can be used to define an action of the symmetric group $S_3$ on the set of vf-safe delta-matroids. 
 The relevance of this for us is that there is a third involution, ${}+e\ast e+e={}\ast e+e\ast e$, which we denote by ${}\barrast e$. 
For a vf-safe delta-matroid $D$ with element $e$, we use 
$D\barrast e$ to denote the delta-matroid 
$D+e\ast e+e$.

\medskip
Bouchet established the relationship between $2$-matroids and delta-matroids in~\cite{MR904585}. Describing and using this relationship places a strain on the notation, so to aid the exposition, we say that a multimatroid with carrier $(U,\Omega)$ is \emph{indexed} by a finite set $E$ if $\Omega=\{\omega_e:e \in E\}$.
\begin{remark}
    Clearly a multimatroid with carrier $(U,\Omega)$ is always indexed by $\Omega$. Nevertheless, we choose to use some other set $E$ as the indexing set as we will use it for the ground set of a delta-matroid associated with the multimatroid and, for clarity, we prefer that the elements 
    of the delta-matroid are not sets of elements of the associated multimatroid.
    This choice becomes even more compelling in Section~\ref{sec:rg}, where we construct a  multimatroid from a ribbon graph. Here, the skew classes are indexed by a set $E$ of discs in a surface and these discs cannot be taken to be the skew classes, which are finite sets. Thus a bijection between the skew classes and discs will always be needed. 
    \expleend 
\end{remark}

Now suppose that we have a $2$-matroid with carrier $(U,\Omega)$ indexed by $E$. Fix a transversal $T$ of $\Omega$. For a subtransversal $S$ of $\Omega$, we let 
\[ T(S) = \{e \in E: T \cap \omega_e = S\cap \omega_e\}.\]
\begin{theorem}[Bouchet~\cite{MR904585}]\label{thm:equivalence}
Let $Z:=(U,\Omega,r)$ be a $2$-matroid indexed by $E$ and let $T$ be a transversal of $\Omega$. Let $\mathcal F := \{T(B):B\in \mathcal B(Z)\}$.
Then $(E,\mathcal F)$ is a delta-matroid.
\end{theorem}
We denote the delta-matroid described in the theorem by $D(Z,T)$. 
The process is reversible and every $2$-matroid may be constructed from a delta-matroid by reversing the 
process. Given a delta-matroid $D:=(E,\mathcal F)$, let $E_1$ and $E_2$ denote two disjoint copies of $E$. Now let $\mu_1$ and $\mu_2$ be bijections from $E$ to $E_1$ and $E_2$, respectively. Then we obtain a $2$-matroid from $D$ indexed by $E$ with carrier $(E_1\cup E_2, \{\omega_e: e\in E\})$, so that
for all $e$ we have $\omega_e=\{ \mu_1(e), \mu_2(e)\}$,
and collection of bases $\{\mu_1(F) \cup \mu_2(E-F) : F \in \mathcal F\}$. 
Given a delta-matroid $D:=(E,\mathcal F)$ we use $Z_{[2]}(D,\mu_1,\mu_2)$ to denote this $2$-matroid. We have $D=D(Z_{[2]}(D,\mu_1,\mu_2),\mu_1(E))$. 

To simplify notation when constructing a \mbox{$2$-matroid} from a delta-matroid $D:=(E,\mathcal F)$, we will always take $E_1:=\dott E:=\{\dott e:e \in E\}$ and $E_2:=\barr E:=\{\barr e:e\in E\}$, $\mu_1(e)=\dott e$ and $\mu_2(e)=\barr e$, and we write $Z_{[2]}(D):=Z_{[2]}(D,\mu_1,\mu_2)$.

For a delta-matroid $D$ with element $e$,
\begin{align*}
Z_{[2]}(D\con e) = Z_{[2]}(D)\res\dott{e}, \quad\text{and} \quad Z_{[2]}(D\ba e) = Z_{[2]}(D)\res\barr{e}.
\end{align*}
Although this was surely known to Bouchet, we cannot find an explicit statement in his work.  However, details may be found in~\cite[Lemma~19]{MR3191496}.

If the cardinalities of the feasible sets in a delta-matroid all have the same parity, then the delta-matroid is \emph{even}.  In~\cite{MR1845490} it was shown that $D$ is an even delta-matroid if and only if $Z_{[2]}(D)$ is tight.

\medskip

More recently,  in~\cite{MR3191496} Brijder and Hoogeboom established the relationship between tight 3-matroids and vf-safe delta-matroids. Following Bouchet~\cite{MM1zbMATH01116184}, given a multimatroid $Z:=(U,\Omega,r)$ and a subset $A$ of its elements, we define $Z[A]$ to be the multimatroid $(A,\Omega',r')$ so that 
\[ \Omega' = \{ \sk{e} \cap A : \sk{e} \in \Omega, \sk{e} \cap A \ne \emptyset\} \]
and $r'$ is the restriction of $r$ to subtransversals of $(A,\Omega')$. Clearly, if $Z$ is a \mbox{$3$-matroid} and $T$ is a transversal of $Z$, then $Z[U(Z)-T]$ is a $2$-matroid. Brijder and Hoogeboom~\cite{MR3191496} proved the following analogue of Theorem~\ref{thm:equivalence}.

\begin{theorem}[Brijder and Hoogeboom~\cite{MR3191496}]\label{thm:3equivalence}
Let $Z:=(U,\Omega,r)$ be a tight $3$-matroid indexed by $E$ and let $T_1$ and $T_2$ be disjoint transversals of $\Omega$. Let $\mathcal F := \{T_1(B):B\in \mathcal B(Z[T_1\cup T_2])\}$.
Then $(E,\mathcal F)$ is a vf-safe delta-matroid.
\end{theorem}
We denote the delta-matroid described in the theorem by $D(Z,T_1,T_2)$. 
Again, as shown by Brijder and Hoogeboom~\cite{MR3191496}, the process is reversible. 
Given a vf-safe delta-matroid $D:=(E,\mathcal{F})$, we obtain its corresponding (indexed) tight \mbox{3-matroid} 
$Z_{[3]}(D)$
as follows. Let $U=\dott{E}\cup\barr{E}\cup\hatt{E}$ and $\Omega=\{\{\dott{e},\barr{e},\hatt{e}\}:e\in E\}$. 
It is not difficult to see that the tightness condition 
implies that for every delta-matroid $D$ on $E$, there is at most one tight $3$-matroid $Z$ with carrier $(U,\Omega)$ such that $Z[\dott E \cup \barr E]=Z_{[2]}(D)$. 
By giving an explicit construction
Brijder and Hoogeboom showed in~\cite{MR3191496} that $Z$ exists if and only if $D$ is vf-safe. When $D$ is vf-safe, we take $Z_{[3]}(D)$ to be this $Z$. Beyond the fact that a transversal $T$ of $Z_{[2]}(D)$ is a basis of $Z_{[3]}(D)$ if and only if it is a basis of $Z_{[2]}(D)$, the precise form of $\mathcal B(Z_{[3]}(D))$ will not concern us. 
Thus $Z_{[2]}(D) = Z_{[3]}(D)[\dott E \cup \barr E]$, so 
\[ D=D(Z_{[3]}(D)[\dott E\cup \barr E],\dott E) = D(Z_{[3]}(D),\dott E,\barr E).\]
Up to isomorphism, every tight 3-matroid arises from a vf-safe delta-matroid in this way. 

\begin{example}
    Let $D$ be the delta-matroid on $\{a,b,c\}$ with collection of feasible sets $\{\{a\},\{b\},\{a,b,c\}\}$. Then its $2$-matroid 
    $Z_{[2]}(D)$ has skew classes $\{\dott a, \barr a\}$, $\{\dott b, \barr b\}$ and $\{\dott c,\barr c\}$. Its bases are $\{\dott a,\barr b, \barr c\}$, $\{\barr a,\dott b,\barr c\}$ and $\{\dott a, \dott b, \dott c\}$. The $3$-matroid $Z_{[3]}(D)$ is the $3$-matroid $Z$ from Example~\ref{example1}.
    \expleend
\end{example}

It has also been shown by Brijder and Hoogeboom in~\cite[Lemma 19]{MR3191496} that for a vf-safe delta-matroid $D:=(E,\mathcal{F})$ with element $e$,  
\begin{equation}\label{eq:z3dmmin}
Z_{[3]}(D\con e) = Z_{[3]}(D)\res \dott{e}, \quad Z_{[3]}(D\ba e) = Z_{[3]}(D)\res \barr{e}, \quad Z_{[3]}(D+ e\con e) = Z_{[3]}(D)\res \hatt{e}.
\end{equation}

Notice that if $D_1$ and $D_2$ are isomorphic delta-matroids, then $Z_{[2]}(D_1)$ and $Z_{[2]}(D_2)$ are isomorphic, as are $Z_{[3]}(D_1)$ and $Z_{[3]}(D_2)$ when $D_1$ and $D_2$ are vf-safe. 
Any isomorphism $\phi:E(D_1)\rightarrow E(D_2)$ leads to an isomorphism $\phi^{[2]}$ of 
$Z_{[2]}(D_1)$ and 
$Z_{[2]}(D_2)$ with $\phi^{[2]}(\dott e) = \dott {({\phi(e)})}$ and $\phi^{[2]}(\barr e) = \barr {\phi(e)}$.
Similarly, when $D_1$ and $D_2$ are vf-safe we get an isomorphism $\phi^{[3]}$ of 
$Z_{[3]}(D_1)$ and 
$Z_{[3]}(D_2)$ with $\phi^{[3]}(\dott e) = \dott {({\phi(e)})}$, $\phi^{[3]}(\barr e) = \barr {\phi(e)}$
and $\phi^{[3]}(\,\hatt e\,) = \hatt {\phi(e)}$. 

But the converses of these assertions are false. For example, if $D_0:=(\{a\},\{\emptyset\})$, $D_1:=(\{a\},\{\{a\}\})$ and $D_2:=(\{a\},\{\{a\},\emptyset\})$, then $Z_{[2]}(D_0) \cong 
Z_{[2]}(D_1)$ and 
$Z_{[3]}(D_0) \cong 
Z_{[3]}(D_1) \cong Z_{[3]}(D_2)$, but $D_0$, $D_1$ and $D_2$ are pairwise non-isomorphic.

Observe that both $\phi^{[2]}$ and $\phi^{[3]}$ are \emph{decoration-preserving} in the sense they map an element with a particular decoration to one with the same decoration. As a result, the expression in Theorem~\ref{thm:rgpdm} for the transition polynomial of the tensor product of two vf-safe delta-matroids is slightly simpler than the corresponding expression for the tensor product of two $3$-matroids.

\begin{proposition}[Bouchet~\cite{MM2zbMATH01119073},  Brijder and Hoogeboom~\cite{MR3191496}]
    Let $D_1$ and $D_2$ be delta-matroids. Then $Z_{[2]}(D_1) \cong Z_{[2]}(D_2)$ if and only if $D_2$ is isomorphic to a twist of $D_1$ and, when $D_1$ and $D_2$ are vf-safe, $Z_{[3]}(D_1) \cong Z_{[3]}(D_2)$ if and only if $D_2$ is isomorphic to a twisted dual of $D_1$
\end{proposition}

The fact that isomorphism of delta-matroids is a finer notion than isomorphism of the corresponding multimatroid will be of relevance to the definition of the tensor product in what follows. Notice that if $D_1$ and $D_2$ are isomorphic then $D_1$ is even if and only if $D_2$ is even, and $D_1$ is vf-safe if and only if $D_2$ is vf-safe. 

To define the two-sum of two delta-matroids, we need the concept of a singular element in a delta-matroid. An element $e$ of a delta-matroid $D$ is \emph{singular} if $e$ is a coloop, $e$ is a loop or $e$ has the property that for every feasible set $F$ of $D$, the set $F\bigtriangleup e$ is also feasible. A singular element of an even delta-matroid is either a coloop or a loop, and $e$ is a singular element of an even delta-matroid $D$ if and only if $\{\dott e,\barr e\}$ is a singular skew class of $Z_{[2]}(D)$. It may also be shown that $e$ is a singular element of a vf-safe delta-matroid $D$ if and only if $\{\dott e, \barr e, \hatt e\}$ is a singular skew class of $Z_{[3]}(D)$. (This is shown in Lemma~19 of~\cite{MR3191496}, although the statement of the result there uses a slightly different definition of minor from what we have used here. See also Remark~20 of~\cite{MR3191496}.) Note that any isomorphism of delta-matroids $D_1$ and $D_2$ maps singular elements of $D_1$ to singular elements of $D_2$ and vice versa. 

We are now ready to give the two-sum operation for delta-matroids, as defined in~\cite{2sumpaper}. First, we deal with even delta-matroids.
Let $D_1:=(E_1,\mathcal F_1)$ and $D_2:=(E_2,\mathcal F_2)$ be two even delta-matroids with $E_1\cap E_2 =\{e\}$, where $e$ is non-singular in both $D_1$ and $D_2$. 
Then, $D_1 \oplus_2 D_2$ has ground set $(E_1 \cup E_2)-\{e\}$ and collection of feasible sets
\[ \{(F_1 \cup F_2) -\{e\} : F_1 \in \mathcal F_1, F_2 \in \mathcal F_2, F_1 \bigtriangleup F_2 = \{e\}\}.\]
Note that~\cite{2sumpaper} 
\[ Z_{[2]}(D_1 \oplus_2 D_2)
= Z_{[2]}(D_1) \oplus_2 Z_{[2]}(D_2).\]
Consequently the two-sum of even delta-matroids is even.

Next, we deal with vf-safe delta-matroids. 
Let $D$ be a delta-matroid with element $e$. 
For a feasible set $F$ of $D$, we say that its \emph{$e$-type} is: \emph{$e$-positive}
 if $e \in F$ and $F-\{e\}$ is not feasible; \emph{$e$-negative} if $e\notin F$ and $F\cup\{e\}$ is not feasible; \emph{$e$-mixed} if $F\bigtriangleup \{e\}$ is feasible.
Let $D_1:=(E_1,\mathcal F_1)$ and $D_2:=(E_2,\mathcal F_2)$ be two vf-safe delta-matroids with $E_1\cap E_2 =\{e\}$, where $e$ is non-singular in both $D_1$ and $D_2$. 
Then, $D_1 \oplus_2 D_2$ has ground set $(E_1 \cup E_2)-\{e\}$ and collection of feasible sets
\[ \{(F_1 \cup F_2) -\{e\} : F_1 \in \mathcal F_1, F_2 \in \mathcal F_2, \text{ the $e$-types of $F_1$ and $F_2$ are different}\}.\]
Again, note that~\cite{2sumpaper} 
\begin{equation}  \label{eq:2sumsdeltavf} Z_{[3]}(D_1 \oplus_2 D_2)
= Z_{[3]}(D_1) \oplus_2 Z_{[3]}(D_2).\end{equation}
Consequently the two-sum of vf-safe delta-matroids is vf-safe.
 
If both $D_1$ and $D_2$ are both even and vf-safe, then we have given two apparently different definitions of the two-sum, but even delta-matroids have no $e$-mixed feasible sets and so the definitions coincide.

\medskip

We can develop the two-sum of a delta-matroid in a similar way to the two-sum of a multimatroid, prove an analogue of Proposition~\ref{prop:sums} and define the graph of a collection of delta-matroids. 
If $\mathcal D$ is a collection of delta-matroids 
whose graph
is a tree, then we can again carry out the two-sums corresponding to the edges of the tree in any order without affecting the result. We denote the resulting delta-matroid by $\bigoplus_2 \mathcal D$. We can also define a star product, but we skip straight to the definition of the tensor product, dealing first with the even case.
Let $D_1$ be an even delta-matroid on $E_1$ in which every element is non-singular. 
Let $\mathcal D:=\{D_e: e \in E_1\}$ be a family of even delta-matroids indexed by $E_1$, having pairwise disjoint ground sets with $E_1 \cap E(D_e) = \{e\}$ for every $e$ in $E_1$. Now let $D_2$ be an even delta-matroid with a non-singular element $a$ and for each $e$ in $E_1$ let $\phi_e$ be an isomorphism from $D_e$ to $D_2$ so that $\phi_e(e)=a$. Let $\Phi:=\{\phi_e:e\in E_1\}$. Then we say that $\bigoplus_2 (\{D_1\} \cup \mathcal D)$ is a \emph{tensor product} of $D_1$ and $D_2$ along $a$ and denote it by $D_1\otimes_{\Phi} D_2$.  In contrast with the situation for multimatroids, a tensor product of $D_1$ and $D_2$ is determined up to isomorphism by $D_1$, $D_2$ and $a$, so we often just denote it by $D_1 \otimes_a D_2$. When $D_1$ and $D_2$ are both vf-safe, then we get the definition of tensor product by replacing the word `even' with `vf-safe' in the above definition. 
When both $D_1$ and $D_2$ are both even and both vf-safe, it follows from the fact that the two definitions of the two-sum for such delta-matroids coincide that the two definitions for the tensor product coincide.

The tensor products of delta-matroids and their associated multimatroids are related as follows.

\begin{proposition}\label{cor:dmtp}
Let $D_1$ be an even delta-matroid on $E_1$ in which every element is non-singular. 
Let $\mathcal D:=\{D_e: e \in E_1\}$ be a family of even delta-matroids indexed by $E_1$, having pairwise disjoint ground sets with $E_1 \cap E(D_e) = \{e\}$ for every $e$ in $E_1$. Now let $D_2$ be an even delta-matroid with a non-singular element $a$ and for each $e$ in $E_1$ let $\phi_e$ be an isomorphism from $D_e$ to $D_2$ so that $\phi_e(e)=a$. Let $\Phi:=\{\phi_e:e\in E_1\}$.
Finally, let $\Phi^{[2]}:=\{\phi^{[2]}_e:e\in E_1\}$. Then,
      \[Z_{[2]}(D_1\otimes_{\Phi}D_2)=  Z_{[2]}(D_1)\otimes_{\Phi^{[2]}} Z_{[2]}(D_2).\]
\end{proposition}

\begin{proposition}\label{cor:vfdmtp}
 Let $D_1$ be a vf-safe delta-matroid on $E_1$ in which every element is non-singular. 
Let $\mathcal D:=\{D_e: e \in E_1\}$ be a family of vf-safe delta-matroids indexed by $E_1$, having pairwise disjoint ground sets with $E_1 \cap E(D_e) = \{e\}$ for every $e$ in $E_1$. Now let $D_2$ be a vf-safe delta-matroid with a non-singular element $a$ and for each $e$ in $E_1$ let $\phi_e$ be an isomorphism from $D_e$ to $D_2$ so that $\phi_e(e)=a$. Let $\Phi:=\{\phi_e:e\in E_1\}$.
Finally, let $\Phi^{[3]}:=\{\phi^{[3]}_e:e\in E_1\}$. Then,
      \[Z_{[3]}(D_1\otimes_{\Phi}D_2)=  Z_{[3]}(D_1)\otimes_{\Phi^{[3]}} Z_{[3]}(D_2).\]
\end{proposition}

\medskip

For a delta-matroid $D$ and subset $X$ of its elements, define
\[d_D(X):=\min_{F\in \mathcal{F}}|F\bigtriangleup X|.\]
When $D$ is clear from context, we omit the subscript.

Let $D$ be a vf-safe delta-matroid on $E$, $\boldsymbol{u}=\{u_e\}_{e\in E}$, $\boldsymbol{v}=\{v_e\}_{e\in E}$, $\boldsymbol{w}=\{w_e\}_{e\in E}$ be indexed families  of formal variables, and $t$ be another formal variable. Then, from~\cite{MR3191496}, the \emph{multivariate transition polynomial} $Q(D;(\boldsymbol{u},\boldsymbol{v},\boldsymbol{w}),t)$ can be defined as
\[Q(D;(\boldsymbol{u},\boldsymbol{v},\boldsymbol{w}),t):=\sum_{(\X,\Y,\Z) \in \mathcal{P}_3(E)}    \Big( \prod_{e\in \X}u_e\Big) \Big(  \prod_{e\in \Y}  v_e\Big) \Big( \prod_{e\in \Z}   w_e\Big)  t^{d_{D\barrast\Z}(\X)},\]
where $\mathcal{P}_3(E)$ denotes the set of ordered partitions of $E$ into three blocks. (The blocks may be empty.)  

Using~\cite[Lemma 18]{MR3191496} it may be shown, for a vf-safe delta-matroid $D$ and partition $(\X,\Y,\Z)\in \mathcal{P}_3$, that $d_{D\barrast \Z}(\X)=n_{Z_{[3]}(D)}(\dott{\X}\cup\barr{\Y}\cup\hatt{\Z})$. So we can obtain the multivariate transition polynomial $Q(D;(\boldsymbol{u},\boldsymbol{v},\boldsymbol{w}),t)$ as an instance of the weighted transition polynomial $Q(Z;\boldsymbol{x},t)$ as follows:
\begin{equation} \label{eq:dmttp}Q(D;(\boldsymbol{u},\boldsymbol{v},\boldsymbol{w}),t)=Q(Z_{[3]}(D);\{x_{\dott e}=u_e, x_{\barr e}=v_e, x_{\hatt e}=w_e\}_{e\in E},t).
\end{equation}
 Suppose that $\phi$ is an isomorphism between delta-matroids $D_1$ and $D_2$. 
It is again convenient to use $\phi: \ints[t,\{u_e,v_e,w_e\}_{e\in E(D_1)}] \rightarrow 
\ints[t,\{u_e,v_e,w_e\}_{e\in E(D_2)}]$ for the ring isomorphism with $\phi(u_e)=u_{\phi(e)}$,
$\phi(v_e)=v_{\phi(e)}$,
$\phi(w_e)=w_{\phi(e)}$, and $\phi(t)=t$.

\begin{theorem}\label{thm:vfsafe2sum}
Let $D_1$ and $D_2$ be vf-safe delta-matroids on $E_1$ and $E_2$, respectively, with $E_1 \cap E_2=\{a\}$, such that $a$ is non-singular in both $D_1$ and $D_2$. Let $E:=(E_1\cup E_2)-\{a\}$ and $E_2':=E_2-\{a\}$.
Then
\[ Q(D_1\oplus_2 D_2; \{u_e,v_e,w_e\}_{e\in E},t) =
Q(D_1;\{u_e,v_e,w_e\}_{e \in E_1},t \mid u_a=\dott y, v_a=\barr y, w_a=\hatt y),\]
where $(\dott y,\barr y, \hatt y)$ is the unique solution to the following system of linear equations in the ring $\ints[t,\{u_e,v_e,w_e\}_{e\in E'_2}]$. 
\begin{align*}
    Q(D_2 \con e; \{u_e,v_e,w_e\}_{e\in E_2'},t) &= t\dott y + \barr y + \hatt y,\\
    Q(D_2 \ba e; \{u_e,v_e,w_e\}_{e\in E_2'},t) &= \dott y + t\barr y + \hatt y,\\
    Q(D_2+e\con e; \{u_e,v_e,w_e\}_{e\in E_2'},t) &= \dott y + \barr y + t\hatt y.
\end{align*}
\end{theorem}
\begin{proof}Let $D:=D_1\oplus_{2} D_2$.
By Equation~\eqref{eq:dmttp},  
\[ Q(D;(\boldsymbol{u},\boldsymbol{v},\boldsymbol{w}),t) = Q(Z_{[3]}(D); \{x_{\dott e} = u_e,
x_{\barr e} = v_e,
x_{\hatt e} = w_e\}_{e\in E},t).\]
By~Equation~\eqref{eq:2sumsdeltavf}, $Z_{[3]}(D)=  Z_{[3]}(D_1)\oplus_{2} Z_{[3]}(D_2)$.
So, by Theorem~\ref{thm:trans2sum}, we have
\begin{multline*}
    Q(Z_{[3]}(D) ; \{x_{\dott e} = u_e,
x_{\barr e} = v_e,
x_{\hatt e} = w_e\}_{e\in E},t)\\
= 
    Q(Z_{[3]}(D_1) ; \{x_{\dott e} = u_e,
x_{\barr e} = v_e,
x_{\hatt e} = w_e\}_{e\in E_1-\{a\}},
x_{\dott a}= \dott y,
x_{\barr a}= \barr y,
x_{\hatt a}= \hatt y,
t),
\end{multline*}
where $(\dott y,\barr y,\hatt y)$ is the unique solution to the following system of linear equations in the ring $\mathbb{Z}[t, \{u_e,v_e,w_e\}_{e\in E'_2}]$.  
   \begin{align}
        Q(Z_{[3]}(D_2)\res\dott a;\{x_{\dott e}=u_e,x_{\barr e}=v_e,x_{\hatt e}=w_e\}_{e\in E'_2},t)=&t\dott y+\barr y+\hatt y,\nonumber \\
        Q(Z_{[3]}(D_2)\res\barr a;\{x_{\dott e}=u_e,x_{\barr e}=v_e,x_{\hatt e}=w_e\}_{e\in E'_2},t)=&\dott y+t\barr y+\hatt y,\label{eq:minors}\\
        Q(Z_{[3]}(D_2)\res\hatt a;\{x_{\dott e}=u_e,x_{\barr e}=v_e,x_{\hatt e}=w_e\}_{e\in E'_2},t)=&\dott y+\barr y+t\hatt y.\nonumber
\end{align}
By applying Equation~\eqref{eq:dmttp} once more, 
we see that
\begin{multline*}  Q(Z_{[3]}(D_1) ; \{x_{\dott e} = u_e,
x_{\barr e} = v_e,
x_{\hatt e} = w_e\}_{e\in E_1-\{a\}},
x_{\dott a}= \dott y,
x_{\barr a}= \barr y,
x_{\hatt a}= \hatt y,
t),\\
= Q(D_1;\{u_e,v_e,w_e\}_{e \in E_1},t \mid u_a=\dott y, v_a=\barr y, w_a=\hatt y).
\end{multline*}
Finally by Equations~\eqref{eq:z3dmmin},~\eqref{eq:dmttp} and~\eqref{eq:minors},  
we observe that
\begin{align*}
    Q(Z_{[3]}(D_2)\res\dott a;\{x_{\dott e}=u_e,x_{\barr e}=v_e,x_{\hatt e}=w_e\}_{e\in {E'_2}},t) &= Q(D_2\con a;(\{u_e,v_e,w_e\}_{e\in{E'_2}},t),\\
    Q(Z_{[3]}(D_2)\res\barr a;\{x_{\dott e}=u_e,x_{\barr e}=v_e,x_{\hatt e}=w_e\}_{e\in {E'_2}},t) &= Q(D_2\ba a;(\{u_e,v_e,w_e\}_{e\in {E'_2}},t),\\
    Q(Z_{[3]}(D_2)\res\hatt a;\{x_{\dott e}=u_e,x_{\barr e}=v_e,x_{\hatt e}=w_e\}_{e\in {E'_2}},t) &= Q(D_2+a\con a;(\{u_e,v_e,w_e\}_{e\in {E'_2}},t),
    \end{align*}
so $\dott y$, $\barr y$ and $\hatt y$ are as described in the theorem statement.
\end{proof}

\begin{theorem}\label{thm:rgpdm}
 Let $D_1$ be a vf-safe delta-matroid on $E_1$ in which every element is non-singular. Let $D_2$ be a vf-safe delta-matroid with non-singular element $a$ and let $E':=E(D_2)-\{a\}$. 
Suppose that $D_1\otimes_{\Phi} D_2$ is a tensor product along $a$ with $\Phi=\{\phi_e:e \in E_1\}$. 
Then
\[Q(D_1\otimes_\Phi D_2;
(\mathbf u, \mathbf v, \mathbf w),t)
=
Q(D_1; \{u_e=\phi_e^{-1}(p), v_e=\phi_e^{-1}(q), w_e=\phi_e^{-1}(r)\}_{e \in E_1}, t),\]
where $(p,q,r)$
is the unique solution to the following system of linear equations in the ring $\ints[t,\{u_e,v_e,w_e\}_{e\in E'}]$.
    \begin{align*}
        Q(D_2\con a;\{u_e,v_e,w_e\}_{e\in E'},t)=&tp+q+r,\\
        Q(D_2\ba a;\{u_e,v_e,w_e\}_{e\in E'},t)=&p+tq+r,\\
        Q(D_2+a\con a;\{u_e,v_e,w_e\}_{e\in E'},t)=&p+q+tr.
    \end{align*}
\end{theorem}
\begin{proof}Let $D:=D_1\otimes_{\Phi} D_2$ and 
$E:=E(D)$. 
By Equation~\eqref{eq:dmttp},  
\[ Q(D;(\boldsymbol{u},\boldsymbol{v},\boldsymbol{w}),t) = Q(Z_{[3]}(D); \{x_{\dott e} = u_e,
x_{\barr e} = v_e,
x_{\hatt e} = w_e\}_{e\in E},t).\]
By Proposition~\ref{cor:vfdmtp}, $Z_{[3]}(D)=  Z_{[3]}(D_1)\otimes_{\Phi^{[3]}} Z_{[3]}(D_2)$.
Recall that $\Phi^{[3]}$ is a family of decoration-preserving maps between the isomorphic copies of $Z_{[3]}(D_2)$ and $Z_{[3]}(D_2)$ itself. So, by Theorem~\ref{thm:main}, we have
\begin{multline*}
    Q(Z_{[3]}(D) ; \{x_{\dott e} = u_e,
x_{\barr e} = v_e,
x_{\hatt e} = w_e\}_{e\in E},t)\\
= Q(Z_{[3]}(D_1);
\{ x_{\dott e} = \phi_e^{-1}(p),
x_{\barr e} = \phi_e^{-1}(q),
x_{\dott e} = \phi_e^{-1}(r)\}_{e\in E_1},t)
\end{multline*} 
where $(p,q,r)$ is the unique solution to the following system of linear equations in the ring $\mathbb{Z}[t, \{u_e,v_e,w_e\}_{e\in E'}]$.  
   \begin{align}
        Q(Z_{[3]}(D_2)\res\dott a;\{x_{\dott e}=u_e,x_{\barr e}=v_e,x_{\hatt e}=w_e\}_{e\in E'},t)=&tp+q+r,\nonumber \\
        Q(Z_{[3]}(D_2)\res\barr a;\{x_{\dott e}=u_e,x_{\barr e}=v_e,x_{\hatt e}=w_e\}_{e\in E'},t)=&p+tq+r,\label{eq:minorstens}\\
        Q(Z_{[3]}(D_2)\res\hatt a;\{x_{\dott e}=u_e,x_{\barr e}=v_e,x_{\hatt e}=w_e\}_{e\in E'},t)=&p+q+tr.\nonumber
\end{align}
By applying Equation~\eqref{eq:dmttp} once more, 
we see that
\begin{multline*} Q(Z_{[3]}(D_1);
\{ x_{\dott e} = \phi_e^{-1}(p),
x_{\barr e} = \phi_e^{-1}(q),
x_{\hatt e} = \phi_e^{-1}(r)\}_{e\in E_1},t)\\
= Q(D_1; \{ u_e = \phi_e^{-1}(p),
v_e = \phi_e^{-1}(q),
w_e = \phi_e^{-1}(r)\}_{e\in E_1},t).
\end{multline*}
Finally by Equations~\eqref{eq:z3dmmin},~\eqref{eq:dmttp} and~\eqref{eq:minorstens},  
we observe that
\begin{align*}
    Q(Z_{[3]}(D_2)\res\dott a;\{x_{\dott e}=u_e,x_{\barr e}=v_e,x_{\hatt e}=w_e\}_{e\in E'},t) &= Q(D_2\con a;(\{u_e,v_e,w_e\}_{e\in E'},t),\\
    Q(Z_{[3]}(D_2)\res\barr a;\{x_{\dott e}=u_e,x_{\barr e}=v_e,x_{\hatt e}=w_e\}_{e\in E'},t) &= Q(D_2\ba a;(\{u_e,v_e,w_e\}_{e\in E'},t),\\
    Q(Z_{[3]}(D_2)\res\hatt a;\{x_{\dott e}=u_e,x_{\barr e}=v_e,x_{\hatt e}=w_e\}_{e\in E'},t) &= Q(D_2+a\con a;(\{u_e,v_e,w_e\}_{e\in E'},t),\\
    \end{align*}
so $p$, $q$ and $r$ are as described in the theorem statement.
\end{proof}

\medskip
A similar analysis gives a tensor product formula for the transition polynomial of Brijder and Hoogeboom from~\cite{MR3191496}. 
Let $D:=(E,\mathcal{F})$ be a delta-matroid, and $w,x,t$  be formal variables. Then the \emph{transition polynomial} $Q(D;w,x,t)$ for delta-matroids is
\[Q(D;w,x,t):= \sum_{A \subseteq E} w^{|E-A|}x^{|A|} t^{d(A)}.\]
(In~\cite{MR3191496} this polynomial was denoted by  $Q_{w,x,0}(D;t)$.)

It can be shown that for every delta-matroid $D$
and transversal $T$ of $Z_{[2]}(D)$, we have
$d_D(T\cap \dott E)=n_{Z_{[2]}(D)}(T)$.
(See Lemma~18 of~\cite{MR3191496} and the comments below it.)
 Consequently it follows that
\begin{equation}\label{eq:dmqtp}
    Q(D;w,x,t)=Q(Z_{[2]}(D);\{x_{\dott{e}}=x, x_{\barr{e}}=w\}_{e\in E},t).
\end{equation}

\begin{theorem}\label{thm:dmeven}
     Let $D_1$ be an even delta-matroid having no non-singular elements and let $D_2$ be an even delta-matroid with non-singular element $e$. Then
     \[Q(D_1\otimes_e D_2;w,x,t)=Q(D_1;p,q,t),\]
     where $(p,q)$ is the unique solution to the following system of linear equations in the ring $\ints[t,w,x]$.
     \begin{align*}
         Q(D_2\con e;w,x,t)&=tp+q,\\
         Q(D_2\ba e;w,x,t)&=p+tq.
     \end{align*}
\end{theorem}
A proof of this result can be obtained following the method used for the proof of Theorem~\ref{thm:rgpdm}.   
 As such we omit its proof. The key points are that Equation~\eqref{eq:dmqtp} and Proposition~\ref{cor:dmtp} are used in place of 
Equation~\eqref{eq:dmttp} and Proposition~\ref{cor:vfdmtp}. 

\medskip

Matroids are precisely delta-matroids in which all feasible sets have the same cardinality. In this situation the feasible sets are exactly the bases of the matroid. 
For a matroid $M$ on $E$, we use $r(M)$ to denote its rank and $n(M)$ to denote its nullity, that is, $|E|-r(M)$.
We use $r(M)$ to denote the rank of a matroid $M$. 
(It will always be clear whether $r$ or $n$ is referring to the rank or nullity of a matroid or of a multimatroid, so the double use of $r$ and $n$ should cause no confusion.)

For an arbitrary  delta-matroid $D:=(E,\mathcal{F})$, let $\mathcal{F}_{\max}$ and $\mathcal{F}_{\min}$ denote the set of feasible sets in $\mathcal{F}$ with maximum and minimum cardinality, respectively. As in~\cite{zbMATH04185622},  $D_{\max}:=(E,\mathcal{F}_{\max})$ and $D_{\min}:=(E,\mathcal{F}_{\min})$ are both matroids. 
By considering the ranks of these matroids we can 
 define a function $\sigma$ on delta-matroids by setting 
\begin{equation}\label{eq:defsigma}
 \sigma(D):=\tfrac{1}{2}(r(D_{\min})+r(D_{\max})).   
\end{equation} 
   For a subset $A$ of $E$, we set $\sigma_D(A):=\sigma(D\backslash A^c)$, where $A^c:=E-A$. We omit the subscript when $D$ is clear from context. Using Proposition~5.39 of~\cite{CMNR1} we can deduce that 
      \begin{equation} \label{eq:siguseful} \sigma_D(A)=\tfrac{1}{2}(|A|-d_D(A)+r(D_{\min})).\end{equation}
(Note that the parameter $\rho_D(A)$ appearing in~\cite{CMNR1} equals $|E|-d_D(A)$ in our notation.)

The \emph{Tutte polynomial}, $R(D;x,y)$, of a delta-matroid $D:=(E,\mathcal{F})$, introduced in~\cite{CMNR1}, is 
\[R(D;x,y):=\sum_{A\subseteq E} (x-1)^{\sigma(E)-\sigma(A)}(y-1)^{|A|-\sigma(A)}.\]

By using Equation~\eqref{eq:defsigma} to expand the exponents, and the definition of $d_D(E)$ for the prefactor, it follows that 
\begin{equation*}\label{eq:dmrgp}
   R(D;x+1,y+1)=\sqrt{x}^{\, r(D_{\max})-|E|}\sqrt{y}^{\,-r(D_{\min})}Q(D;\sqrt{x},\sqrt{y},\sqrt{xy}).
\end{equation*}

\begin{theorem}\label{thm:dmrg}
    Let $D_1$ be an even delta-matroid on $E_1$ having no non-singular elements and let $D_2$ be an even delta-matroid on $E_2$ with non-singular element $e$. 
    Let $D$ denote their tensor product $D_1\otimes_e D_2$, and let $E:=E(D)$.    
Then
    \begin{align*}   R(D;x+1,y+1)=q^{|E_1|}\big(\tfrac{p}{q\sqrt{xy}}\big)^{\sigma(D_1)}\sqrt{x}^{2\sigma(D)-|E|}\sqrt{xy}^{\,r({(D_1)}_{\min})-r(D_{\min})}\\ \cdot R\big(D_1;\tfrac{q\sqrt{xy}}{p}+1,\tfrac{p\sqrt{xy}}{q}+1\big),
    \end{align*}
    where $p$ and $q$ are the unique solutions to the following set of equations in the ring $\ints[x,y]$.
    \begin{align*} 
    R(D_2\con e;x+1,y+1)& = \sqrt{x}^{\,r((D_2\con e)_{\max})-|E_2|+1}\sqrt{y}^{\,-r((D_2\con e)_{\min})}(p\sqrt{xy}+q),\\
    R(D_2\ba e;x+1,y+1)& =  \sqrt{x}^{\,r((D_2\ba e)_{\max})-|E_2|+1}\sqrt{y}^{\,-r((D_2\ba e)_{\min})}(p + q\sqrt{xy} ) .
    \end{align*}
\end{theorem}

\begin{proof}
    Using Equation~\eqref{eq:dmrgp} to express $R(D;x+1,y+1)$ in terms of $Q(D;\sqrt{x},\sqrt{y},\sqrt{xy})$, then applying Theorem~\ref{thm:dmeven} gives
    \begin{align*}
        R(D;x+1,y+1)=\sqrt{x}^{\, r(D_{\max})-|E|}\sqrt{y}^{\,-r(D_{\min})}Q(D_1;q,p,\sqrt{xy}),
    \end{align*}
    where $p$ and $q$ are the unique solutions to the system of linear equations
    \begin{align*}
        Q(D_2\con e;\sqrt{x},\sqrt{y},\sqrt{xy})&=p\sqrt{xy}+q,\\
        Q(D_2\ba e;\sqrt{x},\sqrt{y},\sqrt{xy})&=p+q\sqrt{xy}.
    \end{align*}
    We apply Equation~\eqref{eq:dmrgp} to this system of linear equations to obtain those in the theorem. Using Equation~\eqref{eq:siguseful}, we can rewrite $Q(D_1;q,p,\sqrt{xy})$ as follows
    \begin{align*}
        Q(D_1;q,p,\sqrt{xy})&=\sum_{A\subseteq E_1}q^{|E_1|-|A|}p^{|A|}\sqrt{xy}^{\, d_{D_1}(A)}\\
        &=q^{|E_1|}\sum_{A\subseteq E_1}\big(\tfrac{p}{q}\big)^{|A|}\sqrt{xy}^{\, |A|+r({(D_1)}_{\min})-2\sigma_{D_1}(A)}\\
        &=q^{|E_1|}\sqrt{xy}^{\, r({(D_1)}_{\min})}\sum_{A\subseteq E_1}\big(\tfrac{p\sqrt{xy}}{q}\big)^{|A|-\sigma_{D_1}(A)}\big(\tfrac{q\sqrt{xy}}{p}\big)^{\, -\sigma_{D_1}(A)}\\
        &=q^{|E_1|}\big(\tfrac{p}{q\sqrt{xy}}\big)^{\sigma(D_1)}\sqrt{xy}^{\, r({(D_1)}_{\min})}R(D_1;\tfrac{q\sqrt{xy}}{p}+1,\tfrac{p\sqrt{xy}}{q}+1).
    \end{align*}
    Hence, the result follows.
\end{proof}

Brylawski's original result for the tensor product of the Tutte polynomial of matroids may be obtained as a corollary of Theorem~\ref{thm:dmrg}. 
The \emph{Tutte polynomial}, $T(M;x,y)$, of a matroid $M$ on $E$ is defined as 
\[T(M;x,y):=\sum_{A\subseteq E}(x-1)^{r(E)-r(M\ba A^c)}(y-1)^{|A|-r(M\ba A^c)}.\]
If $M$ is a matroid on $E$
and $A\subseteq E$, then $\sigma(A)=r(M\ba A^c)$ and so $T(M;x,y)=R(M;x,y)$.
Hence,  we obtain the following corollary of Theorem~\ref{thm:dmrg}.
\begin{corollary}\label{cor:brymat}
Let $M_1$ be a loopless, coloopless matroid on $E_1$  and let $M_2$ be a matroid on $E_2$ with element $e$ which is neither a loop nor a coloop.
Then
    \[T(M_1\otimes_e M_2;x,y)=\alpha^{n(M_1)}\beta^{r(M_1)}T(M_1;T(M_2\ba e;x,y)/\beta,T(M_2\con e;x,y)/\alpha),\]
    where $(\alpha,\beta)$ is the unique solution to the following system of linear equations in the ring $\ints[x,y]$. 
    \begin{align*}
        T(M_2\ba e;x,y)&=(x-1)\alpha+\beta,\\
        T(M_2\con e;x,y)&=\alpha+(y-1)\beta.
    \end{align*}
 \end{corollary}

\begin{proof}
    This theorem follows from applying Theorem~\ref{thm:dmrg} to the matroid $M_1\otimes_e M_2$. Write $M:=M_1\otimes_e M_2$ and let $E:=E(M)$.
    We use the observation that for every matroid $N$,  $r({N}_{\max})=r({N}_{\min})=r(N)$ and $\sigma(A)=r(N\ba A^c)$, and that when $e$ is non-singular, $r(M_2\ba e)=r(M_2)$ and $r(M_2\con e)=r(M_2)-1$, to obtain 
    \[T(M;x+1,y+1)=q^{|E_1|}(\tfrac{p}{q})^{r(M_1)}\sqrt{x}^{\, -|E|}\sqrt{\tfrac{x}{y}}^{\,r(M)}T(M_1;\tfrac{q\sqrt{xy}}{p}+1,\tfrac{p\sqrt{xy}}{q}+1),\]
    where $(p,q)$ is the unique solution to the following linear equations in the ring $\ints[x,y]$. 
    \begin{align*}
        T(M_2\con e;x+1,y+1)& = \sqrt{x}^{1-|E_2|}\sqrt{\tfrac{x}{y}}^{\, r(M_2)-1}(p\sqrt{xy}+q),\\
        T(M_2\ba e;x+1,y+1)& =  \sqrt{x}^{1-|E_2|}\sqrt{\tfrac{x}{y}}^{\, r(M_2)}(p + q\sqrt{xy} ).
    \end{align*}
    It is straightforward to check that $r(M)=r(M_1)+|E_1|(r(M_2)-1)$. So by letting $\alpha=q \sqrt{x}^{\, 1-|E_2|}\sqrt{\tfrac{x}{y}}^{\, r(M_2)-1}$ and $\beta=p \sqrt{x}^{\, 1-|E_2|}\sqrt{\tfrac{x}{y}}^{\, r(M_2)}$ we obtain the result.  
\end{proof}

\section{Ribbon graphs}\label{sec:rg}

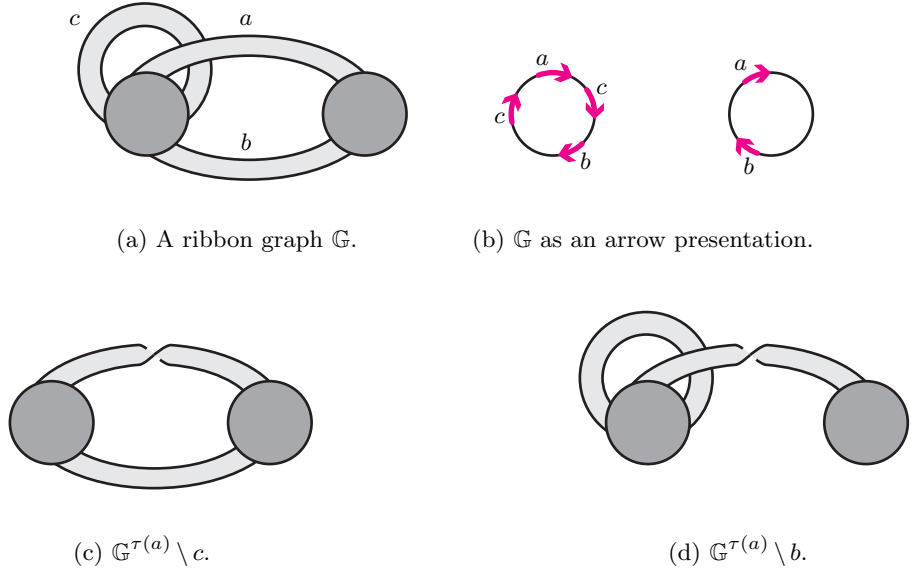
\begin{figure}
     \centering
     \begin{subfigure}[c]{0.4\textwidth}
         \centering
          \labellist
\small\hair 2pt
\pinlabel {$a$} at    75 72
\pinlabel {$b$} at     75 26
\pinlabel {$c$} at     11 72 
\endlabellist
 \includegraphics[scale=1]{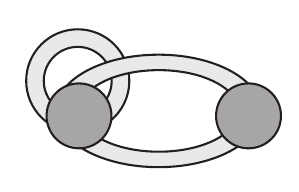}
         \caption{A ribbon graph $\bG$.}
         \label{fig:rg1}
     \end{subfigure}
 \centering
     \begin{subfigure}[c]{0.4\textwidth}
         \centering
          \labellist
\small\hair 2pt
\pinlabel {$a$} at  34 57   
\pinlabel {$a$} at   108 55
\pinlabel {$b$} at   50 19
\pinlabel {$b$} at   111 17
\pinlabel {$c$} at    18 35
\pinlabel {$c$} at   56 46
\endlabellist
 \includegraphics[scale=1]{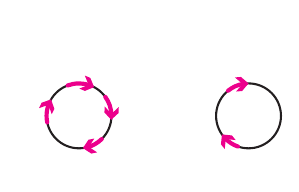}
         \caption{$\bG$ as an arrow presentation.}
         \label{fig:rg4}
     \end{subfigure}  
     \hfill
     \begin{subfigure}[c]{0.4\textwidth}
         \centering
         \includegraphics[scale=1]{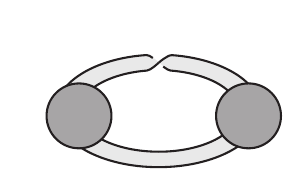}
    \caption{$\bG^{\tau(a)}\ba c$.}
         \label{fig:rg2}
     \end{subfigure}
       \hfill
     \begin{subfigure}[c]{0.4\textwidth}
         \centering
         \includegraphics[scale=1]{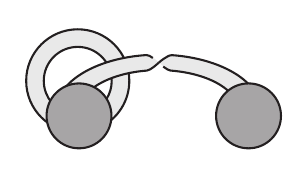}
    \caption{$\bG^{\tau(a)}\ba b$.}
         \label{fig:rg3}
     \end{subfigure}
        \caption{Ribbon graphs.}
        \label{fig:rg}
\end{figure}

 The \emph{topological transition polynomial},  introduced in \cite{MR2869185}, is a multivariate polynomial of ribbon graphs (or equivalently, of graphs embedded in surfaces). It contains both the 2-variable version of  Bollob\'as and Riordan's ribbon graph polynomial~\cite{MR1851080,MR1906909} and the Penrose polynomial~\cite{MR1428870,MR2994409} as specialisations, and is intimately related to  Jaeger's transition polynomial~\cite{MR1096990} and the 
 generalised transition polynomial of 
 \cite{MR1980048}. Jaeger's transition polynomial was formulated in terms of circuit partitions of $4$-regular graphs. For the connections between $4$-regular graphs and ribbon graphs, see~\cite{zbMATH06473572} and the many references therein.

A {\em ribbon graph} $\bG=\left(V,E\right)$ is a surface with boundary, represented as the union of two sets of discs --- a set $V$ of {\em vertices} and a set $E$ of {\em edges} --- such that: (1) the vertices and edges intersect at disjoint line segments; (2) each such line segment lies on the boundary of precisely one vertex and precisely one edge; and (3) every edge contains exactly two such line segments. 
Two ribbon graphs are \emph{isomorphic} if there is a homeomorphism from one to the other that sends vertices to vertices and edges to edges. The homeomorphism should be orientation preserving when the  ribbon graphs are orientable.
As every ribbon graph $\bG$ can be regarded as a surface with boundary, it has some number of boundary components, which we shall denote by $b(\bG)$. The boundary components will be particularly important in what follows. 
 We let $k(\bG)$ denote the number of connected components of $\bG$, and $e(\bG)$ its number of edges. At times we use $E(\bG)$ to denote the edge set of $\bG$. 
Some examples of ribbon graphs can be found in Figure~\ref{fig:rg}. The ribbon graphs in Figures~\ref{fig:rg1} and~\ref{fig:rg2} each have one boundary component, while that in Figure~\ref{fig:rg3} has two.  Additional background on ribbon graphs can be found in~\cite{MR3086663,zbMATH07680505}.

Let $\bG:=(V,E)$ be a ribbon graph and $e\in E$. Then $\bG\ba e$ denotes the ribbon graph obtained from $\bG$ by \emph{deleting} the edge $e$. For $A\subseteq E$, $\bG\ba A$ is the result of deleting each edge in $A$ (in any order).
Edges can also be contracted.  If $u$ and $v$ are vertices incident to $e$ (it is possible that $u=v$ here), consider the boundary component(s) of $e\cup\{u,v\}$ as curves on $\bG$. For each resulting curve, attach a disc (which will form a vertex of $\bG\con e$) by identifying its boundary component with the curve. Delete $e$, $u$ and $v$ from the resulting complex, to get the ribbon graph $\bG\con e$.

The \emph{partial Petrial} with respect to an edge of a ribbon graph $e$,  introduced in \cite{MR2869185} and denoted here by $\bG^{\tau(e)}$,   is, informally, the result  of detaching one end of the edge $e$ from a vertex, giving the edge a half-twist, and reattaching it in the same place. Formally, it is obtained by  detaching an end of $e$ from its incident vertex $v$ creating arcs $[a,b]$ on $v$, and $[a',b']$ on $e$  (so that $\bG$ is recovered by identifying $[a,b]$ with $[a',b']$), then reattaching the end  by identifying the arcs antipodally (so that $[a,b]$ is identified with $[b',a']$). 
 For $A\subseteq E(G)$,  the partial Petrial $\bG^{\tau(A)}$ is the result of forming the partial Petrial with respect to every element in $A$ (in any order).  
 There is a natural correspondence between the edges of $\bG$ and  $\bG^{\tau(A)}$, and so we can and will assume that they have the same edge set. Again, examples can be found in Figure~\ref{fig:rg}.

\medskip

Let $\bG:=(V,E)$ be a ribbon graph, $\boldsymbol{u}=\{u_e\}_{e\in E}$, $\boldsymbol{v}=\{v_e\}_{e\in E}$, $\boldsymbol{w}=\{w_e\}_{e\in E}$ be indexed families  of indeterminates, and $t$  be another indeterminate. Then the
 \emph{topological transition polynomial}, $Q(\bG; (\boldsymbol{u}, \boldsymbol{v}, \boldsymbol{w}) , t)$, can be defined as 
 \[Q(\bG; (\boldsymbol{u}, \boldsymbol{v}, \boldsymbol{w}) , t) :=\sum_{(\X,\Y,\Z) \in \mathcal{P}_3(E)}    \Big( \prod_{e\in \X}u_e\Big) \Big(  \prod_{e\in \Y}  v_e\Big) \Big( \prod_{e\in \Z}   w_e\Big)  t^{b( \bG^{\tau(\Z)}\ba \Y)},\]
where $\mathcal{P}_3(E)$ denotes the set of ordered partitions of $E$ into three blocks (blocks may be empty), and where $b(\bG^{\tau(\Z)}\ba \Y)$ denotes the number of boundary components of $\bG^{\tau(\Z)}\ba \Y$.
Additional background on the topological transition polynomial can be found in, for example,~\cite{zbMATH07680507}.

\begin{example}
If $\bG$ is a plane 2-cycle (i.e., the ribbon graph that is a cycle of length two and has genus 0)  with edges $a$ and $b$ then 
\begin{multline*}
Q(\bG; (\boldsymbol{u}, \boldsymbol{v}, \boldsymbol{w}) , t)
\\=
u_a u_b t^2+
u_a v_b t+
u_a w_b t+
v_a u_b t+
v_a v_b t^2+
v_a w_b t+
w_a u_b t+
w_a v_b t+
w_a w_b t^2.
\end{multline*}\expleend
\end{example}

\medskip

Following~\cite[Section~5]{mmact}, a tight 3-matroid $Z(\bG)$ with carrier $(U, \Omega)$  can be obtained from a  ribbon graph $\bG:=(V,E)$. For this let 
 $U:=\{\dott e, \barr{e}, \hatt e: e\in E\}$, the skew classes be
 $\Omega=\{\{\dott e, \barr{e}, \hatt e \}:e \in E\}$, 
and the set of bases of $Z(\bG)$ be 
\begin{multline*}   \{ 
\{\dott{\x}: \x\in \X\} \cup
\{\barr{\y}: \y\in \Y\}\cup
\{\hatt{\z}: \z\in \Z\} 
: \\ (\X,\Y,\Z)\in\mathcal{P}_3(E)  \text{ and } b(\bG^{\tau(\Z)} \ba \Y)  = k(\bG)
\}.
\end{multline*}
 Following our notational conventions, we denote the skew class $\{\dott e, \barr{e}, \hatt{e}\}$ by $\sk{e}$.

\begin{example}
Consider the ribbon graphs $\bG$ and $\bH$ shown in Figure~\ref{fig:rgs1}. Then the 3-matroid of $\bG$ is given in Example~\ref{example1}, and the 3-matroid of $\bH$ is given in Example~\ref{example1b}.\expleend
\end{example}

We need to recognise which edges of $\bG$ give rise to  singular skew classes of $Z(\bG)$. For this we need some additional ribbon graph terminology.
An edge $e$ of a ribbon graph $\bG$ is a \emph{bridge} if $k(\bG \ba e)>k(\bG)$. It is a \emph{loop} if it is incident to exactly one vertex. A loop $e$ is \emph{orientable} if $e$ together with its incident vertex forms an annulus, and is \emph{nonorientable} if it forms a M\"obius band. A loop $e$ is said to be \emph{interlaced} with a cycle $C$ if when travelling around the boundary of the vertex incident to $e$ we see edges in the cyclic order $e \, c_1\, e \,c_2$ where $c_1$ and $c_2$ are edges of $C$.
A loop is \emph{trivial} if it is not interlaced with any cycle. 
Observe that a bridge necessarily intersects one boundary component of $\bG$, and a trivial orientable loop must meet two boundary components. 

It was shown in~\cite{mmact} that for a ribbon graph $\bG$ with an edge $e$, a skew class $\sk{e}=\{\dott{e},\barr{e},\hatt{e}\}$ of $Z(\bG)$ is singular if and only if the edge $e$ is either a  bridge or trivial loop in $\bG$.
It was also shown there that for each edge $e$ of $\bG$,
\begin{equation}\label{eq:rgzdc}
Z(\bG\con e) =  Z(\bG)\res {\dott{e}}, \quad
 Z(\bG\ba e) =  Z(\bG)\res {\barr{e}}, 
 \quad 
  Z(\bG^{\tau(e)} \con e ) =  Z(\bG)\res {\hatt{e}}.
\end{equation}

As also shown in~\cite{mmact},  transition polynomials for multimatroids and ribbon graphs agree up to a factor of $t^{k(\bG)}$. We have
\begin{equation}\label{eq:rztra}
 Q(\bG; (\boldsymbol{u}, \boldsymbol{v}, \boldsymbol{w}) , t)  = t^{k(\bG)}\cdot  Q(Z(\bG);\vect x, t)  ,
 \end{equation}
when $u_e=x_{\dott{e}}$, $v_e= x_{\barr{e}}$, and $w_e=x_{\hatt{e}}$ for each edge $e$ of $\bG$.

\bigskip

In order to efficiently describe two-sums of ribbon graphs we consider arrow presentations.
An {\em arrow presentation} $\bG$ consists of a set of circles (i.e., closed 1-manifolds) and a set of labels. For each label there are exactly two arrows lying along one or more of the circles. All arrows are disjoint. The set of labels is called the \emph{edge set} and its elements are \emph{edges}.

Arrow presentations describe  ribbon graphs as follows: if $\bG$ is a ribbon graph, for each edge $e$, arbitrarily orient the  boundary of that edge, place an $e$-labelled arrow on the arcs where  $e$ intersects a vertex, pointing in the direction given by the edge's boundary orientation. Taking the boundaries of the vertices together with the labelled arrows gives an arrow presentation. On the other hand, given an arrow presentation, obtain a ribbon graph by identifying each circle with the boundary of a vertex disk.    Then, for each label $e$, take a disc with an orientation of its boundary and identify an arc on its boundary with an $e$-labelled arrow such that the direction of each arrow agrees with the orientation. Thus, in an arrow presentation, each circle corresponds to a vertex of a ribbon graph, and each pair of $e$-labelled arrows corresponds to an edge. 

\medskip

Let $\bG$ and $\bH$ be arrow presentations that have an edge $a$ in common but no other edges in common.
Furthermore, suppose in each arrow presentation that one of the \mbox{$a$-labelled} arrows is distinguished.
Then the \emph{two-sum}  $\bG\oplus_2\bH$ is the arrow presentation constructed by splicing the arrow presentations as indicated in Figure~\ref{fig:ap}. For this, add arcs from:
the tail of the distinguished $a$-labelled arrow in $\bG$ to the head of the distinguished $a$-labelled arrow in $\bH$; 
the head of the distinguished $a$-labelled arrow in $\bG$ to the tail of the distinguished $a$-labelled arrow in $\bH$;
the tail of the non-distinguished $a$-labelled arrow in $\bG$ to the head of the non-distinguished $a$-labelled arrow in $\bH$; 
and
the head of the non-distinguished $a$-labelled arrow in $\bG$ to the tail of the non-distinguished $a$-labelled arrow in $\bH$.
Then delete all four $a$-labelled arrows as well as the arcs of the circles of the arrow presentation that they lie on (excluding the end points of these arcs).

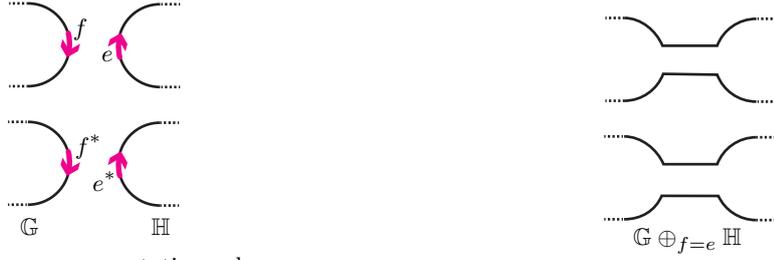
\begin{figure}
     \centering
     \begin{subfigure}[c]{0.4\textwidth}
         \centering
          \labellist
            \small\hair 2pt
            \pinlabel {$a$}  at  27 78
             \pinlabel {$a^*$}  at  30 34
             \pinlabel {$a$}  at  37 68
             \pinlabel {$a^*$}  at 36 21
             \pinlabel {$\bG$}  at  8 4
             \pinlabel {$\bH$}  at  57 4
            \endlabellist
         \includegraphics[scale=1]{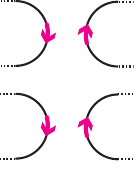}
         \caption{Two arrow presentations shown locally at $a$-labelled arrows.}
         \label{fig:ap1}
     \end{subfigure}
     \hfill
     \begin{subfigure}[c]{0.4\textwidth}
         \centering
          \labellist
            \small\hair 2pt
            \pinlabel {$\bG\oplus_{2} \bH$}  at  31 4
            \endlabellist
         \includegraphics[scale=1]{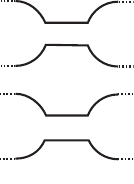}
         \caption{Their two-sum.}
         \label{fig:ap2}
     \end{subfigure}
        \caption{A two-sum of arrow presentations. Distinguished arrows are marked with a $\ast$.}
        \label{fig:ap}
\end{figure}

To define the two-sums of ribbon graphs we suppose that $\bG$ and $\bH$ are ribbon graphs, 
with an edge $a$ in common, but otherwise disjoint. 
 Further  suppose that in each of $\bG$ and $\bH$ that the edge $a$  has a local orientation (i.e., a notion of clockwise for that edge) and is directed (i.e., has a head and a tail). 
Convert $\bG$ and $\bH$ to arrow presentations reading off the directions of the $a$-labelled  arrows so that that they agree with the local edge orientations. Distinguish the  arrow that arose from the head of $a$ in each ribbon graph. Form the two-sum of the arrow presentations, then convert back to ribbon graphs. 

\begin{example}
Figures~\ref{fig:rgs1} and~\ref{fig:rgs2} shows two ribbon graphs $\bG$ and $\bH$ and their arrow presentations. Their two sum is shown as an arrow presentation in Figure~\ref{fig:rgs3} and as a ribbon graph in Figure~\ref{fig:rgs5}. Let $\bH'$ be the same as $\bH$ but with the edge named $a$ now named $c$. Then $\bG \oplus_2 \bH'$ (formed with respect to $c$) is shown as an arrow presentation in Figure~\ref{fig:rgs4} and as a ribbon graph in Figure~\ref{fig:rgs5}. \expleend
\end{example}

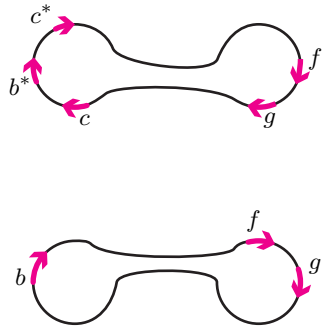
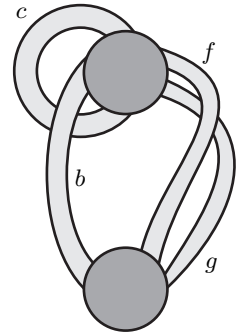
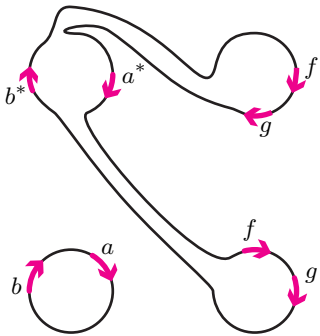
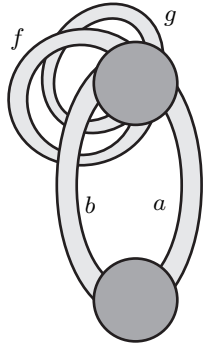
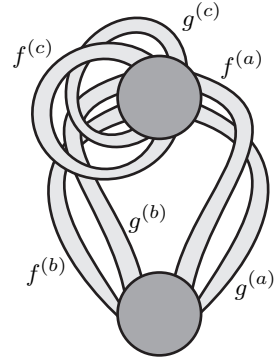
\begin{figure}
     \centering
     \begin{subfigure}[c]{0.4\textwidth}
         \centering
          \labellist
            \small\hair 2pt
            \pinlabel {$a$}  at  58 66
             \pinlabel {$b$}  at   30 66
             \pinlabel {$c$}  at   6 128
             \pinlabel {$a$}  at 104 66
             \pinlabel {$f$}  at   155 90
             \pinlabel {$g$}  at  155 35
             \pinlabel {$\bG$}  at  46 2 
             \pinlabel {$\bH$}  at   115 2
            \endlabellist
         \includegraphics[scale=1]{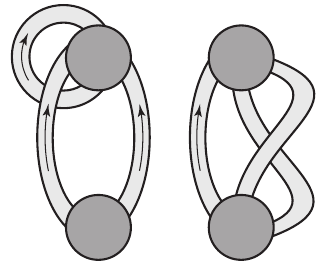}
         \caption{Two ribbon graphs. For two-sums and tensors, the edges are oriented anticlockwise (not shown on the figure).}
         \label{fig:rgs1}
     \end{subfigure}
   \hfill
     \begin{subfigure}[c]{0.4\textwidth}
         \centering
          \labellist
             \small\hair 2pt
            \pinlabel {$a$}  at   36 40
             \pinlabel {$b$}  at    2 26
             \pinlabel {$c$}  at  26 85  
             \pinlabel {$a$}  at  72 31
             \pinlabel {$f$}  at  89 47  
             \pinlabel {$g$}  at  112 30 
              \pinlabel {$a^*$}  at 46 106  
             \pinlabel {$b^*$}  at  2 98  
             \pinlabel {$c^*$}  at  10 125  
             \pinlabel {$a^*$}  at  73 92
             \pinlabel {$f$}  at    112 108
             \pinlabel {$g$}  at   95 85
             \pinlabel {$\bG$}  at  21 2 
             \pinlabel {$\bH$}  at   90 2
            \endlabellist
         \includegraphics[scale=1]{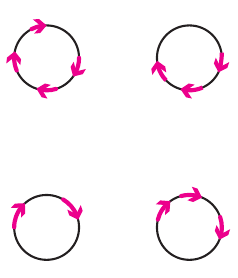}
         \caption{The ribbon graphs as arrow presentations.}
         \label{fig:rgs2}
     \end{subfigure}
   \hfill
     \begin{subfigure}[c]{0.4\textwidth}
         \centering
          \labellist
             \small\hair 2pt
             \pinlabel {$b$}  at    2 26
             \pinlabel {$c$}  at  26 85  
             \pinlabel {$f$}  at  89 47  
             \pinlabel {$g$}  at  112 30 
             \pinlabel {$b^*$}  at  2 98  
             \pinlabel {$c^*$}  at  10 125  
             \pinlabel {$f$}  at    112 108
             \pinlabel {$g$}  at   95 85
            \endlabellist
         \includegraphics[scale=1]{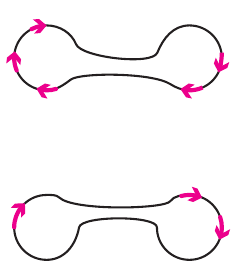}
         \caption{$\bG\oplus_{2}\bH$.}
         \label{fig:rgs3}
     \end{subfigure}
           \hfill
         \begin{subfigure}[c]{0.4\textwidth}
         \centering
          \labellist
            \small\hair 2pt
             \pinlabel {$b$}  at   32 66
             \pinlabel {$c$}  at   10 128
             \pinlabel {$f$}  at   80 113 
             \pinlabel {$g$}  at  81 33
            \endlabellist
         \includegraphics[scale=1]{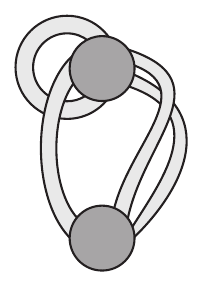}
         \caption{$\bG\oplus_{2}\bH$.}
         \label{fig:rgs5}
     \end{subfigure}
        \hfill
     \begin{subfigure}[c]{0.3\textwidth}
         \centering
          \labellist
             \small\hair 2pt
             \pinlabel {$a$}  at   36 40
             \pinlabel {$b$}  at    2 26
             \pinlabel {$f$}  at  89 47  
             \pinlabel {$g$}  at  112 30 
              \pinlabel {$a^*$}  at 46 106  
             \pinlabel {$b^*$}  at  2 98  
             \pinlabel {$f$}  at    112 108
             \pinlabel {$g$}  at   95 85
            \endlabellist
         \includegraphics[scale=1]{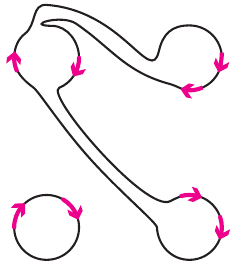}
         \caption{$\bG\oplus_{2}\bH'$.}
         \label{fig:rgs4}
     \end{subfigure}
        \hfill
         \begin{subfigure}[c]{0.3\textwidth}
         \centering
          \labellist
            \small\hair 2pt
             \pinlabel {$b$}  at   32 60
             \pinlabel {$a$}  at   58 60
             \pinlabel {$f$}  at   5 122 
             \pinlabel {$g$}  at  62 130
            \endlabellist
         \includegraphics[scale=1]{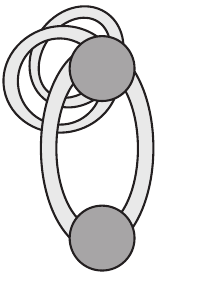}
         \caption{$\bG\oplus_{2}\bH'$.}
         \label{fig:rgs6}
     \end{subfigure}
           \hfill
          \begin{subfigure}[c]{0.3\textwidth}
         \centering
          \labellist
            \small\hair 2pt
             \pinlabel {$f_a$}  at 80 116 
             \pinlabel {$g_a$}  at   85 37 
                \pinlabel {$f_b$}  at  7 40 
             \pinlabel {$g_b$}  at 43 60
                 \pinlabel {$f_c$}  at 2 122  
             \pinlabel {$g_c$}  at  62 134
            \endlabellist
         \includegraphics[scale=1]{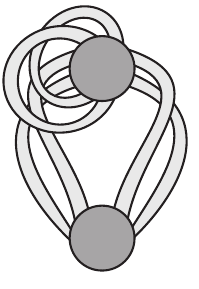}
         \caption{$\bG\otimes_{\Phi}\bH$.}
         \label{fig:rgs7}
     \end{subfigure}
         \caption{Forming two-sums and tensor products of ribbon graphs.}
        \label{fig:rgs}
\end{figure}

As shown in~\cite{serpar} the two-sum operation for ribbon graphs and for multimatroids are compatible operations.

\begin{theorem}\label{thm:rgtp}
Let $\bG$ and $\bH$ be ribbon graphs, with a common edge $a$ that is locally oriented and directed in both $\bG$ and $\bH$, and no other edges in common. Further, suppose that in both $\bG$ and $\bH$, $a$ is neither a bridge nor a trivial loop. Then
\[ Z(\bG\oplus_{2}\bH) = Z(\bG)\oplus_{2}Z(\bH)  .\]
\end{theorem}
We note that while the particular ribbon graph $\bG\oplus_{2}\bH$ constructed as a two-sum depends upon the choices of local orientation and direction on the common edge, the 3-matroid $Z(\bG\oplus_{2}\bH)$ does not.

We can use two-sums to define tensor products. 
Let $\bG$ be a ribbon graph with no bridges or trivial loops
in which every edge has a local orientation and is directed. 
Let $\mathcal H=\{\bH_e:e \in E(\bG)\}$ be a family of pairwise disjoint ribbon graphs indexed by the edges of $\bG$, so that for each edge $e$ of $\bG$,  
the ribbon graphs $\bG$ and $\bH_e$ have the edge $e$ in common but are otherwise disjoint. Moreover, suppose that in $\bH_e$, $e$ has a local orientation and is directed. 
Further, suppose that $\bH$ is a ribbon graph with a distinguished edge $a$ that is neither a bridge nor a trivial loop and has a local orientation and is directed,
and that for each edge $e$ of $\bG$, there is an isomorphism $\phi_e:\bG_e\rightarrow \bH$ with $\phi_e(e)=a$, preserving the local orientation and the direction of $e$. 
Let $\Phi=\{\phi_e:e\in E(\bG)\}$.
Then the \emph{tensor product} $\bG\otimes_{\Phi} \bH $ is the ribbon graph obtained by successively forming the two-sum of $\bG$ and each $\bH_e$.

\begin{example}
Take $\bG$ and $\bH$ as in Figure~\ref{fig:rgs1}. Now let $\bH_a$, $\bH_b$ and $\bH_c$ be isomorphic copies of $\bH$,
with edge sets $\{a,f_a,g_a\}$, $\{b,f_b,g_b\}$ and $\{c,f_c,g_c\}$ respectively. For each edge $e$ of $\bG$, let $\phi_e:\bH_e\rightarrow \bH$ be the isomorphism with $\phi_e(e)=a$, $\phi_e(f_e)=f$ and $\phi_e(g_e)=g$, and let $\Phi=\{\phi_e,\phi_f,\phi_g\}$. Then
$\bG \otimes_{\Phi} \bH$ is shown in Figure~\ref{fig:rgs7}.\expleend
\end{example}

Analogously with the situation in delta-matroids,
every isomorphism $\phi$ of a ribbon graph leads to an isomorphism $\phi^{[3]}$ of its $3$-matroid, with
$\phi^{[3]}(\dott e) = \dott {({\phi(e)})}$, $\phi^{[3]}(\barr e) = \barr {\phi(e)}$
and $\phi^{[3]}(\,\hatt e\,) = \hatt {\phi(e)}$. 
As a corollary of Theorem~\ref{thm:rgtp} we have the following.
\begin{corollary}\label{cor:rgtp}
Let $\bG$, $\bH$ and $\Phi$ be as above and let $\Phi^{[3]}:= \{\phi_e^{[3]}:e\in E(\bG)\}$.
Then
\[ Z(\bG\otimes_{\Phi}\bH) = Z(\bG)\otimes_{\Phi^{[3]}}Z(\bH)  .\]
 \end{corollary}

We can now obtain an expression for the transition polynomial of the tensor product of two ribbon graphs.
\begin{theorem}\label{thm:rgtpf}
Let $\bG$, $\bH$ and $\Phi$ be as above and let $E':=E(\bH)-\{a\}$. Then 
\begin{multline*}
Q(\bG \otimes_\Phi \bH;
(\mathbf u, \mathbf v, \mathbf w),t)
\\= t^{e(\bG)(k(\bH)-1)}
Q(\bG; \{u_e=\phi_e^{-1}(\dott y), v_e=\phi_e^{-1}(\barr y), w_e=\phi_e^{-1}(\hatt y)\}_{e \in E'}, t),\end{multline*}
where $(\dott{y}, \barr{y}, \hatt{y})$ is the 
 unique solution to the following system of linear equations in the ring
 $\ints[t,\{u_e,v_e,w_e\}_{e\in E'}]$.
    \begin{align*}
        Q(\bH \con a; (\boldsymbol{u}, \boldsymbol{v}, \boldsymbol{w}) , t) &=t^{k(\bH)} (t\dott{y}+\barr{y}+\hatt{y}), \\
          Q(\bH \ba a; (\boldsymbol{u}, \boldsymbol{v}, \boldsymbol{w}) , t) &=t^{k(\bH)} (\dott{y}+t\barr{y}+\hatt{y}), \\
          Q(\bH^{\tau(a)} \con a; (\boldsymbol{u}, \boldsymbol{v}, \boldsymbol{w}) , t) &=t^{k(\bH)} (\dott{y}+\barr{y}+t\hatt{y}).
    \end{align*}
\end{theorem}
\begin{proof}
The proof is similar to the proof of Theorem~\ref{thm:rgpdm}, with the minor additional complication of dealing with the prefactor, so we just outline how this affects the proof. Let $E:=E(\bG\otimes_{\Phi}\bH)$. We have $k(\bG\otimes_{\Phi} \bH) = k(\bG) + e(\bG)(k(\bH)-1)$. So in the first step we apply 
Equation~\eqref{eq:rztra} to get
\begin{multline*}
Q(\bG \otimes_\Phi \bH;
(\mathbf u, \mathbf v, \mathbf w),t)
\\=
t^{k(\bG) + e(\bG)(k(\bH)-1)}Q(Z(\bG \otimes_\Phi \bH); \{x_{\dott e} = u_e,
x_{\barr e} = v_e,
x_{\hatt e} = w_e\}_{e\in E},t).\end{multline*}
We then apply Corollary~\ref{cor:rgtp} to express this is in terms of $Q(Z(\bG)\otimes_{\Phi} Z(\bH))$, before applying Theorem~\ref{thm:main}, and then Equation~\eqref{eq:rztra} once again to return to ribbon graphs. In the last step we lose a factor of $t^{k(\bG)}$ from the prefactor.

The system of linear equations arises from those of Theorem~\ref{thm:main}, by again using Equation~\eqref{eq:rztra} to express them in terms of ribbon graphs,  then applying Equations~\eqref{eq:rgzdc}. 
The rewriting also uses the fact that $k(\bH) = k(\bH \con e) = k(\bH \ba e)=k(\bH^{\tau(e)} \con e)$, which follows because $e$ is neither a bridge nor a trivial loop.
\end{proof}

Note that by taking $\bG$ and $\bH$ to be connected, $e$ to be a non-loop edge, and each variable in $\boldsymbol{u}$, $\boldsymbol{v}$, and $\boldsymbol{w}$ to be identical we recover exactly Theorem~6.1 of~\cite{ELLIS_MONAGHAN_2014}. Our proof here shows that the requirement for $e$ to be a non-loop edge in that theorem can be weakened.

As in~\cite[Section~6]{ELLIS_MONAGHAN_2014}, Theorem~\ref{thm:rgtpf} can be specialised to give a tensor product formula for the Tutte polynomial of a ribbon graph, which is also known as the ribbon graph polynomial or the 2-variable Bollob\'as--Riordan (an extension of the formula to the full 3-variable Bollob\'as--Riordan polynomial can be found in~\cite{maya2}).
Alternatively, this tensor product formula can be deduced from Theorem~\ref{thm:rgpdm} by considering the  delta-matroid $D(\bG)$ of $\bG$, following the construction in~\cite[Section~4]{CMNR1} and making use of~\cite[Theorems~6.4 and~6.6]{CMNR1} to relate the ribbon graph and delta-matroid Tutte polynomials.

\bigskip 

\noindent\large{{\bf Acknowledgements}} We thank Donggyu Kim for their helpful comments. We thank the referees for their considered comments. 

\bigskip

\noindent\large{{\bf Funding Information}} Iain Moffatt and Steven Noble were supported by the Engineering and Physical Sciences Research Council [grant numbers EP/W033038/1 and EP/W032945/1, respectively].

\bigskip 

\noindent\large{{\bf Declarations}} 

\bigskip

\noindent\large{{\bf Conflict of interest}} On behalf of all authors, the corresponding author states that there is no conflict of interest.

\bigskip 

\noindent\large{{\bf Data}} No underlying data is associated with this article.

\bigskip 

\noindent\large{{\bf Open Access}} For the purpose of open access, the authors have applied a Creative Commons
Attribution (CC BY) licence to any Author Accepted Manuscript version arising.

\bibliography{multimatroid-tens}

@phdthesis{geelen-thesis,
    author = {Geelen, James F.},
    title = {Matchings, matroids and unimodular matrices},
    school = {University of Waterloo},
    year = {1995}
}

@article{zbMATH00750738,
 author = {Bouchet, Andr{\'e} and Cunningham, William H.},
 title = {Delta-matroids, jump systems, and bisubmodular polyhedra},
 fjournal = {SIAM Journal on Discrete Mathematics},
 journal = {SIAM J. Discrete Math.},
 issn = {0895-4801},
 volume = {8},
 number = {1},
 pages = {17--32},
 year = {1995},
 language = {English},
 doi = {10.1137/S0895480191222926},
 zbMATH = {750738},
 Zbl = {0821.05010}
}

@article{zbMATH01355177,
 author = {Bouchet, Andr{\'e} and Cunningham, W. H. and Geelen, J. F.},
 title = {Principally unimodular skew-symmetric matrices},
 fjournal = {Combinatorica},
 journal = {Combinatorica},
 issn = {0209-9683},
 volume = {18},
 number = {4},
 pages = {461--486},
 year = {1998},
 language = {English},
 doi = {10.1007/s004930050033},
 zbMATH = {1355177},
 Zbl = {0924.05010}
}

@Article{zbMATH06381893,
 Author = {Traldi, Lorenzo},
 Title = {Binary matroids and local complementation},
 FJournal = {European Journal of Combinatorics},
 Journal = {Eur. J. Comb.},
 ISSN = {0195-6698},
 Volume = {45},
 Pages = {21--40},
 Year = {2015},
 Language = {English},
 DOI = {10.1016/j.ejc.2014.10.001},
 zbMATH = {6381893},
 Zbl = {1304.05015}
}

@Article{zbMATH05700345,
 Author = {Traldi, Lorenzo},
 Title = {Weighted interlace polynomials},
 FJournal = {Combinatorics, Probability and Computing},
 Journal = {Comb. Probab. Comput.},
 ISSN = {0963-5483},
 Volume = {19},
 Number = {1},
 Pages = {133--157},
 Year = {2010},
 Language = {English},
 DOI = {10.1017/S0963548309990381},
 zbMATH = {5700345},
 Zbl = {1216.05059}
}

@Article{zbMATH01018470,
 Author = {Bouchet, Andr{\'e} and Ghier, Laurence},
 Title = {Connectivity and {{\(\beta\)}}-invariants of isotropic systems and 4-regular graphs},
 FJournal = {Discrete Mathematics},
 Journal = {Discrete Math.},
 ISSN = {0012-365X},
 Volume = {161},
 Number = {1-3},
 Pages = {25--44},
 Year = {1996},
 Language = {English},
 DOI = {10.1016/0012-365X(95)00219-M},
 zbMATH = {1018470},
 Zbl = {0870.05013}
}

@Misc{zbMATH00015493,
 Author = {Bouchet, Andr{\'e}},
 Title = {Connectivity of isotropic systems},
 Year = {1989},
 Language = {English},
 HowPublished = {Combinatorial mathematics, {Proc}. 3rd {Int}. {Conf}., {New} {York}/ {NY} ({USA}) 1985, {Ann}. {N}. {Y}. {Acad}. {Sci}. 555, 81-93},
 zbMATH = {15493},
 Zbl = {0741.05047}
}

@Article{zbMATH03689441,
 Author = {Cunningham, William H. and Edmonds, Jack},
 Title = {A combinatorial decomposition theory},
 FJournal = {Canadian Journal of Mathematics},
 Journal = {Can. J. Math.},
 ISSN = {0008-414X},
 Volume = {32},
 Pages = {734--765},
 Year = {1980},
 Language = {English},
 DOI = {10.4153/CJM-1980-057-7},
 zbMATH = {3689441},
 Zbl = {0442.05054}
}

@Article{zbMATH03783038,
 Author = {Cunningham, William H.},
 Title = {Decomposition of directed graphs},
 FJournal = {SIAM Journal on Algebraic and Discrete Methods},
 Journal = {SIAM J. Algebraic Discrete Methods},
 ISSN = {0196-5212},
 Volume = {3},
 Pages = {214--228},
 Year = {1982},
 Language = {English},
 DOI = {10.1137/0603021},
 zbMATH = {3783038},
 Zbl = {0497.05031}
}

@article {MR3577648,
    AUTHOR = {Brijder, Robert and Traldi, Lorenzo},
     TITLE = {Isotropic matroids {I}: {M}ultimatroids and neighborhoods},
   JOURNAL = {Electron. J. Combin.},
  FJOURNAL = {Electronic Journal of Combinatorics},
    VOLUME = {23},
      YEAR = {2016},
    NUMBER = {4},
     PAGES = {Paper 4.1, 41},
      ISSN = {1077-8926},
   MRCLASS = {05B35 (05C31 05C50)},
  MRNUMBER = {3577648},
MRREVIEWER = {Gary\ Gordon},
       DOI = {10.37236/5222},
       URL = {https://doi.org/10.37236/5222},
}

@article {MR919874,
    AUTHOR = {Bouchet, Andr\'e},
     TITLE = {Isotropic systems},
   JOURNAL = {European J. Combin.},
  FJOURNAL = {European Journal of Combinatorics},
    VOLUME = {8},
      YEAR = {1987},
    NUMBER = {3},
     PAGES = {231--244},
      ISSN = {0195-6698,1095-9971},
   MRCLASS = {05B35 (05C99)},
  MRNUMBER = {919874},
MRREVIEWER = {Daniel\ Turz\'ik},
       DOI = {10.1016/S0195-6698(87)80027-6},
       URL = {https://doi.org/10.1016/S0195-6698(87)80027-6},
}

@article{MR1428870,
	author = {Aigner, Martin},
	doi = {10.1007/s002080050030},
	fjournal = {Mathematische Annalen},
	issn = {0025-5831},
	journal = {Math. Ann.},
	mrclass = {05C15},
	mrnumber = {1428870},
	mrreviewer = {W. T. Tutte},
	number = {2},
	pages = {173--189},
	title = {The {P}enrose polynomial of a plane graph},
	volume = {307},
	year = {1997}}

@misc{zbMATH04162893,
	author = {Bouchet, A.},
	howpublished = {Combinatorics, {Proc}. 7th {Hung}. {Colloq}., {Eger}/{Hung}. 1987, {Colloq}. {Math}. {Soc}. {J{\'a}nos} {Bolyai} 52, 167-182 (1988).},
	language = {English},
	title = {Representability of {{\(\Delta\)}}-matroids},
	year = {1988},
	zbl = {0708.05013},
	zbmath = {4162893}}

@article{zbMATH04185622,
	author = {Bouchet, Andr{\'e}},
	doi = {10.1016/0012-365X(89)90161-1},
	fjournal = {Discrete Mathematics},
	issn = {0012-365X},
	journal = {Discrete Math.},
	language = {English},
	number = {1-2},
	pages = {59--71},
	title = {Maps and {{\(\Delta\)}}-matroids},
	volume = {78},
	year = {1989},
	zbl = {0719.05019},
	zbmath = {4185622}}

@article{MM1zbMATH01116184,
	author = {Bouchet, Andr{\'e}},
	doi = {10.1137/S0895480193242591},
	fjournal = {SIAM Journal on Discrete Mathematics},
	issn = {0895-4801},
	journal = {SIAM J. Discrete Math.},
	language = {English},
	number = {4},
	pages = {626--646},
	title = {Multimatroids. {I}: {Coverings} by independent sets},
	volume = {10},
	year = {1997},
	zbl = {0886.05042},
	zbmath = {1116184}}

@article{MM2zbMATH01119073,
	author = {Bouchet, Andr{\'e}},
	fjournal = {The Electronic Journal of Combinatorics},
	issn = {1077-8926},
	journal = {Electron. J. Comb.},
doi ={https://doi.org/10.37236/1346},
	language = {English},
	number = {1},
	pages = {research paper r8, 25},
	title = {Multimatroids. {II}: {Orthogonality}, minors and connectivity},
	volume = {5},
	year = {1998},
	zbl = {0885.05050},
	zbmath = {1119073}}

@incollection{MR1845490,
	author = {Bouchet, Andr\'{e}},
	doi = {10.1006/eujc.2000.0486},
	fjournal = {European Journal of Combinatorics},
	issn = {0195-6698},
	journal = {European J. Combin.},
	mrclass = {05B35},
	mrnumber = {1845490},
	mrreviewer = {Gary Gordon},
	note = {Combinatorial geometries (Luminy, 1999)},
	number = {5},
	pages = {657--677},
	title = {Multimatroids. {III}. {T}ightness and fundamental graphs},
	volume = {22},
	year = {2001}}

@article{MR904585,
	author = {Bouchet, Andr\'{e}},
	doi = {10.1007/BF02604639},
	fjournal = {Mathematical Programming},
	issn = {0025-5610},
	journal = {Math. Programming},
	mrclass = {05B35 (90C27)},
	mrnumber = {904585},
	mrreviewer = {Ulrich Faigle},
	number = {2},
	pages = {147--159},
	title = {Greedy algorithm and symmetric matroids},
	volume = {38},
	year = {1987}}

@article{zbMATH05982480,
	author = {Brijder, Robert and Hoogeboom, Hendrik Jan},
	doi = {10.1016/j.ejc.2011.03.002},
	fjournal = {European Journal of Combinatorics},
	issn = {0195-6698},
	journal = {Eur. J. Comb.},
	language = {English},
	number = {8},
	pages = {1353--1367},
	title = {The group structure of pivot and loop complementation on graphs and set systems},
	volume = {32},
	year = {2011},
	zbl = {1230.05197},
	zbmath = {5982480}}

@article{MR1851080,
	author = {Bollob{\'a}s, B{\'e}la and Riordan, Oliver},
	doi = {10.1112/plms/83.3.513},
	fjournal = {Proceedings of the London Mathematical Society. Third Series},
	issn = {0024-6115},
	journal = {Proc. Lond. Math. Soc. (3)},
	language = {English},
	number = {3},
	pages = {513--531},
	title = {A polynomial invariant of graphs on orientable surfaces},
	volume = {83},
	year = {2001},
	zbl = {1015.05024},
	zbmath = {1696342}}

@article{MR1906909,
	author = {Bollob{\'a}s, B{\'e}la and Riordan, Oliver},
	doi = {10.1007/s002080100297},
	fjournal = {Mathematische Annalen},
	issn = {0025-5831},
	journal = {Math. Ann.},
	language = {English},
	number = {1},
	pages = {81--96},
	title = {A polynomial of graphs on surfaces},
	volume = {323},
	year = {2002},
	zbl = {1004.05021},
	zbmath = {1801590}}

@article{MR2994409,
	author = {Ellis-Monaghan, Joanna A. and Moffatt, Iain},
	doi = {10.1016/j.ejc.2012.06.009},
	fjournal = {European Journal of Combinatorics},
	issn = {0195-6698},
	journal = {European J. Combin.},
	mrclass = {05C31 (05C10)},
	mrnumber = {2994409},
	number = {2},
	pages = {424--445},
	title = {A {P}enrose polynomial for embedded graphs},
	volume = {34},
	year = {2013}}

@InCollection{zbMATH07680505,
 Author = {Chmutov, Sergei},
 Title = {Topological extensions of the {Tutte} polynomial},
 BookTitle = {Handbook of the {T}utte polynomial and related topics},
 Pages = {497--513},
 Year = {2022},
	publisher = {CRC Press},
    address = {Boca Raton, FL},
DOI = {10.1201/9780429161612-27},
}

@InCollection{zbMATH07680507,
 Author = {Ellis-Monaghan, Joanna A. and Moffatt, Iain},
 Title = {Skein polynomials and the {Tutte} polynomial when {{\(x = y\)}}},
 BookTitle = {Handbook of the Tutte polynomial and related topics},
 ISBN = {978-1-4822-4062-7; 978-1-032-23193-8; 978-0-429-16161-2},
 Pages = {266--283},
 Year = {2022},
publisher = {CRC Press},
    address = {Boca Raton, FL},
 DOI = {10.1201/9780429161612-13},
 zbMATH = {7680507},
 Zbl = {1512.05215}
}

@article{MR3191496,
	author = {Brijder, R. and Hoogeboom, H. J.},
	doi = {10.1016/j.ejc.2014.03.005},
	fjournal = {European Journal of Combinatorics},
	issn = {0195-6698},
	journal = {European J. Combin.},
	mrclass = {05B35},
	mrreviewer = {Stefan H. M. van Zwam},
	pages = {142--167},
	title = {Interlace polynomials for multimatroids and delta-matroids},
	volume = {40},
	year = {2014}}

@InCollection{Brylawski_2010,
 Author = {Brylawski, Thomas},
 Title = {The {Tutte} polynomial. {I}: {General} theory},
 Edition = {Reprint of the 1982 original published by {Liguori}, {Napoli} and {Birkh{\"a}user}},
 BookTitle = {Matroid theory and its applications. Lectures given at a summer school of the Centro Internazionale Matematico Estivo (C.I.M.E.), held in Varenna, Italy, August 24 -- September 2, 1980},
 ISBN = {978-3-642-11109-9; 978-3-642-11110-5},
 Pages = {125--275},
 Year = {2010},
 Publisher = {Springer},
 address = {Berlin},
 DOI = {10.1007/978-3-642-11110-5_3}
}

@book{MR3086663,
	author = {Ellis-Monaghan, Joanna A. and Moffatt, Iain},
	doi = {10.1007/978-1-4614-6971-1},
	isbn = {978-1-4614-6970-4; 978-1-4614-6971-1},
	pages = {xii+139},
	publisher = {Springer},
    address = {New York},
	series = {SpringerBriefs in Mathematics},
	title = {Graphs on surfaces: polynomials, dualities and knots},
	year = {2013}}

@article{CMNR2,
	author = {Chun, Carolyn and Moffatt, Iain and Noble, Steven D. and Rueckriemen, Ralf},
	doi = {10.1112/plms.12190},
	fjournal = {Proceedings of the London Mathematical Society. Third Series},
	issn = {0024-6115},
	journal = {Proc. Lond. Math. Soc. (3)},
	language = {English},
	number = {3},
	pages = {675--700},
	title = {On the interplay between embedded graphs and delta-matroids},
	url = {eprints.bbk.ac.uk/id/eprint/23754/1/interplay_final_v1.pdf},
	volume = {118},
	year = {2019},
	zbl = {1410.05029},
	zbmath = {7055741}}

@article{CMNR1,
	author = {Chun, Carolyn and Moffatt, Iain and Noble, Steven D. and Rueckriemen, Ralf},
	doi = {10.1016/j.jcta.2019.02.023},
	fjournal = {Journal of Combinatorial Theory. Series A},
	issn = {0097-3165},
	journal = {J. Comb. Theory, Ser. A},
	language = {English},
	pages = {7--59},
	title = {Matroids, delta-matroids and embedded graphs},
	url = {eprints.bbk.ac.uk/id/eprint/26891/1/delta-ribbon-final.pdf},
	volume = {167},
	year = {2019},
	zbl = {1417.05103},
	zbmath = {7094555}}

@Article{zbMATH00008503,
 Author = {Bouchet, A. and Duchamp, A.},
 Title = {Representability of {{\(\bigtriangleup\)}}-matroids over {{\(GF(2)\)}}},
 FJournal = {Linear Algebra and its Applications},
 Journal = {Linear Algebra Appl.},
 ISSN = {0024-3795},
 Volume = {146},
 Pages = {67--78},
 Year = {1991},
 Language = {English},
 DOI = {10.1016/0024-3795(91)90020-W},
 zbMATH = {8503},
 Zbl = {0738.05027}
}

@Article{zbMATH04081574,
 Author = {Bouchet, Andr{\'e}},
 Title = {Graphic presentations of isotropic systems},
 FJournal = {Journal of Combinatorial Theory. Series B},
 Journal = {J. Comb. Theory, Ser. B},
 ISSN = {0095-8956},
 Volume = {45},
 Number = {1},
 Pages = {58--76},
 Year = {1988},
 Language = {English},
 DOI = {10.1016/0095-8956(88)90055-X},
 zbMATH = {4081574},
 Zbl = {0662.05014}
}

@article{MR1980048,
	author = {Ellis-Monaghan, Joanna A. and Sarmiento, Irasema},
	fjournal = {Congressus Numerantium},
	issn = {0384-9864},
	journal = {Congr. Numerantium},
	language = {English},
	pages = {57--69},
	title = {Generalized transition polynomials},
	volume = {155},
	year = {2002},
	zbl = {1021.05067},
	zbmath = {1933265}}

@article{ELLIS_MONAGHAN_2014,
 Author = {Ellis-Monaghan, J. and Moffatt, I.},
 Title = {Evaluations of topological {Tutte} polynomials},
 FJournal = {Combinatorics, Probability and Computing},
 Journal = {Comb. Probab. Comput.},
 ISSN = {0963-5483},
 Volume = {24},
 Number = {3},
 Pages = {556--583},
 Year = {2015},
 Language = {English},
 DOI = {10.1017/S0963548314000571},
 zbMATH = {6785442},
 Zbl = {1371.05134}
}

@article{Huggett_2011,
	author = {Stephen Huggett and Iain Moffatt},
	doi = {10.1007/s00026-011-0116-3},
	journal = {Annals of Combinatorics},
	month = {oct},
	number = {4},
	pages = {675--706},
	publisher = {Springer Science and Business Media {LLC}},
	title = {Expansions for the {B}ollob{\'{a}}s-{R}iordan Polynomial of Separable Ribbon Graphs},
	url = {https://doi.org/10.1007%2Fs00026-011-0116-3},
	volume = {15},
	year = 2011}

@unpublished{serpar,
	author = {Criel Merino and Iain Moffatt and Steven Noble},
	note = {in preparation},
	title = {Series–parallel delta-matroids}}

@article{mmact,
 author = {Merino, Criel and Moffatt, Iain and Noble, Steven},
 title = {An activities expansion of the transition polynomial of a multimatroid},
 fjournal = {SIAM Journal on Discrete Mathematics},
 journal = {SIAM J. Discrete Math.},
 issn = {0895-4801},
 volume = {39},
 number = {2},
 pages = {1372--1407},
 year = {2025},
 language = {English},
 doi = {10.1137/23M1549468},
 zbMATH = {8061623},
 Zbl = {1577.05025}
}

@unpublished{maya,
 Author = {Moffatt, Iain and Thompson, Maya},
 Title = {Deletion-contraction and the surface {Tutte} polynomial},
 FJournal = {European Journal of Combinatorics},
 Journal = {Eur. J. Comb.},
 ISSN = {0195-6698},
 Volume = {118},
 Pages = {20},
 Note = {Id/No 103933},
 Year = {2024},
 Language = {English},
 DOI = {10.1016/j.ejc.2024.103933},
 zbMATH = {7824158},
 Zbl = {1535.05148}
}

@Article{zbMATH06473572,
 Author = {Traldi, Lorenzo},
 Title = {The transition matroid of a 4-regular graph: an introduction},
 FJournal = {European Journal of Combinatorics},
 Journal = {Eur. J. Comb.},
 ISSN = {0195-6698},
 Volume = {50},
 Pages = {180--207},
 Year = {2015},
 Language = {English},
 DOI = {10.1016/j.ejc.2015.03.016},
 zbMATH = {6473572},
 Zbl = {1319.05034}
}

@phdthesis{maya2,
  author  = {Maya Thompson},
  title   = {Topological Analogues of the {T}utte polynomial and their Decompositions},
  school  = {Royal Holloway, University of London},
  year    = {2004},
url = {https://pure.royalholloway.ac.uk/en/publications/topological-analogues-of-the-tutte-polynomial-and-their-decomposi}
}

@unpublished{2sumpaper,
	author = {Steven Noble and Irene Pivotto and Gordon Royle},
	note = {In preparation.},
	title = {Constructions for delta-matroids and multimatroids}}

@incollection{MR1096990,
	author = {Jaeger, F.},
	booktitle = {Cycles and rays ({M}ontreal, {PQ}, 1987)},
	mrclass = {05C15 (57M25)},
	mrnumber = {1096990 (92c:05059)},
	mrreviewer = {W. T. Tutte},
	pages = {123--150},
	publisher = {Kluwer Acad. Publ.},
    address = {Dordrecht},
	series = {NATO Adv. Sci. Inst. Ser. C Math. Phys. Sci.},
	title = {On transition polynomials of {$4$}-regular graphs},
	volume = {301},
	year = {1990},
	doi = {doi.org/10.1007/978-94-009-0517-7\textunderscore 12}}

@article{MR2869185,
	author = {Ellis-Monaghan, Joanna A. and Moffatt, Iain},
	doi = {10.1090/S0002-9947-2011-05529-7},
	fjournal = {Transactions of the American Mathematical Society},
	issn = {0002-9947},
	journal = {Trans. Amer. Math. Soc.},
	mrclass = {05C10 (05C25)},
	mrnumber = {2869185},
	mrreviewer = {Seyed Amin Seyed Fakhari},
	number = {3},
	pages = {1529--1569},
	title = {Twisted duality for embedded graphs},
	volume = {364},
	year = {2012}}

@book{MR1207587,
	author = {Oxley, James G.},
	edition = {2nd ed.},
	fseries = {Oxford Graduate Texts in Mathematics},
	isbn = {978-0-19-856694-6; 978-0-19-960339-8},
	language = {English},
	publisher = {Oxford University Press},
        address = {Oxford},
	series = {Oxf. Grad. Texts Math.},
	title = {Matroid theory},
	volume = {21},
	year = {2011},
doi = {https://doi.org/10.1093/acprof:oso/9780198566946.001.0001},
	zbl = {1254.05002},
	zbmath = {5873618}}

@article{zbMATH03985246,
	author = {Dress, Andreas and Havel, Timothy F.},
	doi = {10.1016/0001-8708(86)90104-0},
	fjournal = {Advances in Mathematics},
	issn = {0001-8708},
	journal = {Adv. Math.},
	language = {English},
	pages = {285--312},
	title = {Some combinatorial properties of discriminants in metric vector spaces},
	volume = {62},
	year = {1986},
	zbl = {0609.05029},
	zbmath = {3985246}}

@article{zbMATH04070920,
	author = {Chandrasekaran, R. and Kabadi, Santosh},
	doi = {10.1016/0012-365X(88)90101-X},
	fjournal = {Discrete Mathematics},
	issn = {0012-365X},
	journal = {Discrete Math.},
	language = {English},
	number = {3},
	pages = {205--217},
	title = {Pseudomatroids},
	volume = {71},
	year = {1988},
	zbl = {0656.05023},
	zbmath = {4070920}}

\end{document}